\documentclass[11pt]{article}
\usepackage[utf8]{inputenc}
\usepackage{amsmath,amsfonts,amssymb}
\usepackage{graphicx}
\usepackage{xcolor}
\definecolor{upforestgreen}{rgb}{0.0, 0.35, 0.12}
\usepackage[colorlinks=true,linkcolor=blue,citecolor=blue,urlcolor=black]{hyperref} 
\usepackage[margin=1in]{geometry}
 \usepackage{amsthm}
\usepackage{times}
\usepackage{comment}
\usepackage{sectsty}
\usepackage{lmodern}
\usepackage{tikz}
\usetikzlibrary{calc}
\usepackage[light]{sourcesanspro}
\usepackage{natbib}

\newcommand\blfootnote[1]{%
  \begingroup
  \renewcommand\thefootnote{}\footnote{#1}%
  \addtocounter{footnote}{-1}%
  \endgroup
}

\title{\fontsize{25}{28.2}{A Distributionally Robust Estimator that\\Dominates the Empirical Average}}

\author{%
    Nikolas Koumpis~~ and~~ Dionysis Kalogerias \\[5pt]
    Yale University
}

\date{}

\newtheorem{theorem}{Theorem}
\newtheorem{informal theorem}{Informal Theorem}
\newtheorem{assumption}{Assumption}
\newtheorem{corollary}{Corollary}
\newtheorem{proposition}{Proposition}
\newtheorem{lemma}{Lemma}

\allowdisplaybreaks

\begin{document}

\maketitle
\begin{abstract}
\blfootnote{Contact: \texttt{{\{nikolaos.koumpis,dionysis.kalogerias\}@yale.edu}}}
We leverage the duality between risk-averse and distributionally robust optimization (DRO) to devise a distributionally robust estimator that strictly outperforms the empirical average for all probability distributions with negative excess kurtosis. The aforesaid estimator solves the $\chi^{2}-$robust mean squared error problem in closed form.
\end{abstract}

\section{Introduction}
The aim of this paper is to deploy the basic insights of DRO in order to devise an empirical estimator that outperforms the (commonly used) sample average w.r.t.\ the mean squared error; the latter being a standard figure of merit for measuring losses in statistics and machine learning \citep{ girshick1951bayes, james1961estimation, devroye2003estimation, wang2009mean, verhaegen2007filtering,kuhn2019wasserstein}. 
\par To set up the stage for our investigation, let $\mathrm{B}(\mathbb{R}^d)$ be the set of Borel probability measures on $\mathbb{R}^d$ and $D=\{X_{1},\dots,X_{n}\}$ a set of independent identically distributed (i.i.d.) random vectors realizing $X\sim \mu \in \mathrm{B}(\mathbb{R}^d)$. The corresponding to $D$ empirical measure is $\mu_{n}=n^{-1}\sum^{n}_{k=1}\delta_{X_{i}}$, where $\delta_{x}$ is the Delta measure of unit mass located at $x\in \mathbb{R}^{d}$. 
\par By defining the real-valued map $f(\xi,\widehat{\mathrm{X}})\equiv\mathbb{E}_{\xi}\Vert \widehat{\mathrm{X}}-X\Vert^{2}_{2},$ $\xi\in \mathrm{B}(\mathbb{R}^{d})$, we obtain the following characterization of the empirical average
\vspace{-3pt}
\begin{align}
\alpha(D;\mu_{n})=\dfrac{1}{n}\sum^{n}_{i}X_{i}=\underset{\widehat{\mathrm{X}}\in \mathbb{R}^{d}}{\mathrm{argmin}}f(\mu_{n},\widehat{\mathrm{X}}), \label{ea}
\end{align}
as the solution to the surrogate of the ``true'' problem
\begin{align}
\min_{\widehat{\mathrm{X}}\in\mathbb{R}^{d}}f(\mu,\widehat{\mathrm{X}}),
\label{ideal}
\end{align}
when $\mu$ is unknown. Plausibly then, we are interested in the performance of (\ref{ea}) on $\mu$:
\begin{align}
\pi_{\alpha}(\mu,\mu_{n})\equiv f(\mu,\alpha(D;\mu_{n}))-f^{\star},
\end{align}
where\footnote{We assume that the first order moment is finite} $f^{\star}=f(\mu,\mathbb{E}X)$ is the optimal value of (\ref{ideal}). Although (\ref{ea}) dominates
among unbiased estimators for normally distributed low dimensional data, in the case of skewed or heavy tailed measures, it exhibits a rather sub-optimal performance and therefore, it leaves room for improvement. Such an improvement could be potentially achieved by an appropriate re-weighting $\nu_{n}\in \mathrm{B}(\mathbb{R}^d)$, $\nu_{n}\ll\mu_{n}$ on $D$, such that 
\begin{align}
\pi_{\alpha}(\mu,\nu_{n})\leq \pi_{\alpha}(\mu,\mu_{n}).
\end{align}
The most prominent approach for doing so is via distributionally robust optimization. As a paradigm, DRO \citep{scarf1957min} has gained interest due to its connections to regularization, generalization, robustness and statistics \citep{staib2019distributionally}, \citep{blanchet2021statistical}, \citep{bertsimas2004price}, \citep{blanchet2021statistical}, \citep{blanchet2023statistical}, \citep{nguyen2023bridging},\citep{blanchet2024distributionally}. 
\par Let $\mathrm{F}(\beta,\rho)$ be the field of DRO problem instances $\mathrm{P}$ of the form 
\begin{align}
\min_{\widehat{\mathrm{X}}\in\mathbb{R}^{d}}\bigg\{\sup_{\xi\ll\mu}  f(\xi,\widehat{\mathrm{X}})~:~{\xi \in \beta_{\rho}\left(\mu\right)} \bigg\}, \label{dro0}
\end{align}
where $\beta_{\rho}(\mu)\subseteq \mathrm{B}(\mathbb{R}^{d})$ is the uncertainty set of radius $\rho\geq 0$ centered at $\mu$. The two main approaches for representing the uncertainty set, recently unified by \citep{blanchet2023unifying}, are the divergence approach, and the Wasserstein approach. In the former one, distributional shifts are measured in terms of likelihood ratios~\citep{bertsimas2004price,bayraksan2015data,namkoong2016stochastic,duchi2018learning,van2021data}, while in the latter, the uncertainty set is represented by a Wasserstein ball~\citep{esfahani2015data,gao2023distributionally,lee2018minimax,kuhn2019wasserstein}. 
\par Let $\widehat{\mathrm{X}}^{\star}$ be a solution of $\mathrm{P}\in \mathrm{F}(\beta,\rho)$. Based on the discussion so far, we are interested in an ideal sub-field of instances $\mathrm{F}^{\star}(\beta,\rho)\subset \mathrm{F}(\beta,\rho)$ formed by the following design specifications: First, if $\mathrm{P}\in \mathrm{F}^{\star}(\beta,\rho)$, then there exist a saddle point $(\widehat{\mathrm{X}}^{\star}(\xi^{\star}),\xi^{\star}(\widehat{\mathrm{X}}^{\star}))$. Then, the first design specification implies that necessarily
\begin{align}
\widehat{\mathrm{X}}^{\star}=\widehat{\mathrm{X}}^{\star}(\xi^{\star})=\mathbb{E}_{\xi^{\star}(\widehat{\mathrm{X}}^{\star})}X.
\end{align}
In some cases, robustness against distributional shifts is as safeguarding against risky events (naturally coordinated by the distribution tail). That is, solving a DRO problem is equivalent to solving a risk-averse problem. This brings us to the second specification which requires $\widehat{\mathrm{X}}^{\star}$
to be the closed form solution of a risk-averse problem (over $\mu$). Combining the aforesaid specifications would allow us to \textbf{{implement}} the empirical counterpart $\alpha(D;\xi^{\star}_{\rho,n})$ of $\widehat{\mathrm{X}}^{\star}$ from data $D$ drawn from $\mu$. That being said, in this paper we study the following motivating question:
\begin{center}
\textit{Is there an instance $\mathrm{P}\in\mathrm{F}^{\star}(\beta,\rho)$ such that $\pi_{\alpha}(\mu,\xi^{\star}_{\rho,n}(\beta))< \pi_{\alpha}(\mu,\mu_{n})$?}
\end{center}
The aforesaid dual relation between risk-averse and DRO suggests a place to start our design. In particular, our methodology is as follows: We are going to state a risk-averse problem (which in principle is easier to design), that is also solvable in closed form. Then, we are going to show that under certain conditions, its solution also solves a certain, precisely defined DRO problem. Then, we are going to verify our initial assertion by studying the performance of the corresponding empirical solution. We start by considering the following risk-constrained extension of (\ref{ideal}):
\begin{equation}
\begin{array}{cl}
\underset{\widehat{\mathrm{X}}\in \mathbb{R}^{d}}{\operatorname{min}}&f(\mu,\widehat{\mathrm{X}}) \\
\text { s.t. } & \mathbb{E}_{\mu}\big(\Vert{ X}-\widehat{\mathrm{X}}\Vert_2^2~-~\mathbb{E}_{\mu}\Vert{X}-\widehat{\mathrm{X}}\Vert_2^2\big)^2  \leq \varepsilon.
\end{array} 
\label{rfp0}
\end{equation}
The risk-averse problem (\ref{rfp0}) is a quadratically constrained quadratic program and admits the closed form solution (see Section \ref{sec2})
\begin{align}
\widehat{ \mathrm{X}}^{\star}_{\lambda^{\star}}{=}\big(I+4\lambda^{\star}\Sigma\big)^{-1}\Big(\mathbb{E} X+2\lambda^{\star}\big(\mathbb{E}[\Vert X\Vert^2_{2} X]-\mathbb{E}\Vert X\Vert^{2}_{2}\mathbb{E} X\big)\Big), \label{b}
\end{align}
where $\lambda^{\star}\geq {0}$. Risk-averse optimization refers to the optimization of risk measures. Unlike expectations, these are (often) convex functionals on appropariate spaces of random variables and deliver optimizers that take into account the risky events of the underlying distribution. Unlike expectations, a convex risk measure can be represented through the Legendre transform as the supremum of its affine minorants. This representation essentially characterizes the risk measure as a supremum over a certain class of affine functionals, specifically over a subset of the dual space of signed measures \citep{shapiro2021lectures}. For the particular class of coherent risk measures, this set contains probability measures and thus, safeguarding for risky events is as being robust against distributional shifts. In other words, solving a risk-constrained problem on the population at hand is as solving a min-max problem where expectations are taken for an appropriate re-weighting of that population.
\par That said, the challenge is that minimization of a risk measure does not always correspond to a DRO formulation. As an example, problem (\ref{rfp0}) is closely related to the mean-variance risk-measure \citep[p. 275]{shapiro2021lectures}, which does not lead to a DRO formulation in general. Mean-variance optimization was initially utilized by~\citep{e5a1bb8f-41b7-35c6-95cd-8b366d3e99bc}, in portfolio selection. Later on, \citep{abeille2016lqg} deploys mean-variance for dynamic portfolio allocation problems, and \citep{kalogerias2020better} stressed the importance of safeguarding against risky events in mean squared error Bayesian estimation. DRO with $\chi^2$-divergence is almost equivalent as controlling by the variance \citep{gotoh2018robust, lam2016robust,duchi2019variance} were the worst case measure is computed exactly in $O(n\log(n))$ \citep{staib2019distributionally}. Although minimization of the mean-variance risk-measure does not correspond to an element of $\mathrm{F}(\beta,\rho)$, the mean deviation risk measure of order $p=2$ does (see Theorem \ref{Th}).
\par At this point, we state informally the first result of this paper (see Theorem \ref{Th}), which establishes that the \textbf{candidate} problem (\ref{rfp0}) corresponds to an instance $\mathrm{P}\in\mathrm{F}^{\star}(\beta,\rho)$: 
\begin{informal theorem}\label{Th000}
There exists constant $\lambda^{\star}>0$ such that, for every $\lambda\leq \lambda^{\star}$, $f$ has a saddle point  and (\ref{b}) solves the $\chi^{2}-$DRO problem 
\begin{align}
\underset{\widehat{{\mathrm{X}}}\in\mathbb{R}^{d}}{\mathrm{min}}~\bigg\{ \sup_{\nu\ll\mu}~\mathbb{E}_{ \nu}\Vert \widehat{\mathrm{X}}-X\Vert^{2}_{2}~:~\chi^{2}(\nu||\mu)\leq{\lambda}\bigg\}. \label{dro000}
\end{align}
\end{informal theorem}
As discussed later, Theorem \ref{Th} states that for the corresponding levels of risk $\varepsilon$ of (\ref{rfp0}), the closed form solution of (\ref{rfp0}) solves all $\mathrm{P}\in \mathrm{F}(\chi^{2},\lambda^{\star})$.
We answer the main question of the paper in the one dimensional case (see Theorem \ref{theorem}). We find that in the one dimensional case, the empirical counterpart of (\ref{b}) dominates the empirical average for all probability distributions with negative excess kurtosis, $\kappa$.
\begin{informal theorem} \label{theorem}Let $\widehat{\mathrm{X}}^{\star}_{n,\bar{\lambda}}$ be the empirical counterpart of (\ref{b}), and $\bar{ \mathrm{X}}_{n}$ be the empirical average. For all platykurtic probability distributions, there is  $\lambda_{n}=O(|\kappa|/n)$ such that, for every $n\geq 3$ and every $\bar{\lambda}<\lambda_{n}$,
\begin{align}
\mathbb{E}|\widehat{\mathrm{X}}^{\star}_{n,\bar{\lambda}}- \mathbb{E}X|^2<\mathbb{E}|\bar{\mathrm{X}}_{n}- \mathbb{E}X|^2. \label{gn}
\end{align}
\end{informal theorem}

\begin{figure}[h!]
\centering
    \includegraphics[scale=0.22]{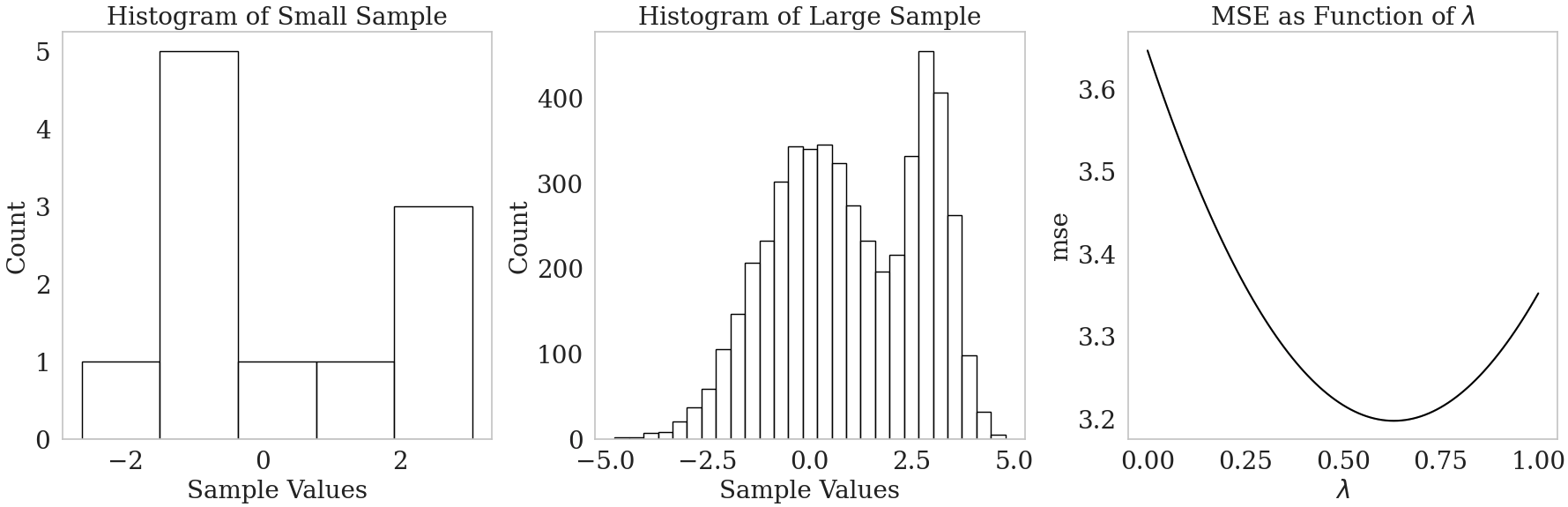}
    \caption{Gaussian mixture model of two mixands with parameters $\mathrm{mean}_1=3, \mathrm{mean}_2=$ $0.1, \sigma_1=0.5, \sigma_2=1.7$, and mixing proportions $\pi_1=0.3, \pi_2=1-\pi_1$. The condition of Theorem \ref{theorem} is satisfied with $3 m_2^2-m_4=11.4622$. Specifically, the variance of the mixture is calculated to be $m_2=3.826$, and the fourth central moment is approximately $m_4=32.46$. We draw $n=23$ i.i.d. samples (left) and we test $\widehat{\mathrm{X}}_{23, \bar{\lambda}}^{\star}=1.963-\bar{\lambda} 3.61$, against the nominal population simulated with $n=10000$ samples (middle) w.r.t.\ the mse (right) $\operatorname{mse}(\bar{\lambda})=\mathbb{E}_\mu(1.963-\bar{\lambda} 3.61-X)^2$.}
    \label{fig:enter-label}
\end{figure}
All proofs of the results presented hereafter can be found in the supplementary material.

\newpage
\section{Risk-constrained minimum mean squared error estimators}\label{sec2}
We begin analyzing the risk-constrained mmse problem (\ref{rfp0}) by first referring the reader to \citep[Lemma 1]{kalogerias2019risk}, where the well-definiteness of the statement (\ref{rfp0}) is ensured under the following regularity condition
\begin{assumption} \label{as1}
It is true that $\int \|{ x}\|_2^3~\mathrm{d}\mu( x)<+\infty\text {.}$   
\end{assumption}

\par By noticing that the constraint in (\ref{rfp0}) can be written as a perfect square, we establish a quadratic reformulation that on the one hand, allows us to interpret the solution to (\ref{rfp0}) as a stereographic projection, and on the other, sets the stage for a derivative-free solution: 
\begin{lemma}[The risk-constrained estimator as stereographic projection] \label{lem1}
Under Assumption \ref{as1}, problem (\ref{rfp0}) is well-defined and equivalent to the convex quadratic program
\begin{equation}
\begin{array}{cl}
\underset{\widehat{ \mathrm{X}}\in\mathbb{R}^{d}}{\mathrm{min}} & \mathbb{E}_{\mu}\|{ X}-\widehat{ \mathrm{X}}\|^{2}_{2}\\
 \mathrm { s.t. } & \|\widehat{ \mathrm{X}}-\frac{1}{4}\Sigma^{-1}\zeta^{-}\|^{2}_{4\Sigma}\le\bar{\varepsilon}
\end{array}\label{rfp2}
\end{equation}
where
$\zeta^{-}=2\big(\mathbb{E}[\|{ X}\|^{2}_{2}{ X}]-\mathbb{E}\|{ X}\|^{2}_{2}\mathbb{E} X
\big)$, $\bar{\varepsilon}=\varepsilon+(\zeta^{-})^{\top}\frac{\Sigma^{-1}}{4}\zeta^{-}-\alpha$, $\alpha=\mathbb{E}(|| X||^{2}-\mathbb{E}|| X||^{2})^2$ and ${\Sigma}\succcurlyeq{0}$ is the covariance matrix of $ X$ \footnote{By $\Vert\cdot\Vert_{Q}$ we mean the quadratic form $x^{\top}Qx$ w.r.t.\ the positive semi-definite matrix $Q\in\mathbb{R}^{n\times n}$.}.
\end{lemma}
Lemma \ref{lem1} shows the equivalence of (\ref{rfp0}) to the convex quadratic program (\ref{rfp2}) allowing to interpret the former in geometric terms based on the standard inner product in $\mathbb{R}^d$. Within this setting, the solution to (\ref{rfp0}) is a stereographic projection of the $\mathrm{mmse}-$sphere onto $\mathbb{R}^{d}$. Further, we can do slightly more by applying Pythagora's theorem to the objective in (\ref{rfp2}) to obtain 
\begin{equation}
\begin{array}{cl}
\underset{\widehat{ \mathrm{X}}\in{\mathbb{R}^d}}{\mathrm{min}} & \Vert\widehat{ \mathrm{X}}-\mathbb{E} X\Vert^{2}_{2}\\
\mathrm{s.t.} & 
\Vert\widehat{ \mathrm{X}}-\frac{1}{4}{\Sigma}^{-1}\zeta^{-}\Vert^{2}_{4{\Sigma}}\leq{\bar{\varepsilon}},
\end{array},\label{rfp3}
\end{equation}
where by $\zeta^{-}$ we mean the opposite vector to $\zeta$. According to Lemma \ref{lem1}, the solution of (\ref{rfp3}) (and therefore of (\ref{rfp0})) is the point lying on the ellipse of radius $\sqrt{\bar{\varepsilon}}$ and of center $\frac{1}{4}{\Sigma}^{-1}\zeta$, with the least distance from the orthogonal projection $\mathbb{E} X$.
Being guided by the geometric picture, the problem is feasible for any $\bar{\varepsilon}\geq{0}$. In particular, for all  $\bar{\varepsilon}\geq \bar{\varepsilon}_{c}=\Vert{\mathbb{E} X-\frac{1}{4}{\Sigma}^{-1}\zeta^{-}\Vert^{2}_{4{\Sigma}}}$, $\widehat{ \mathrm{X}}^{*}=\mathbb{E} X$, while for $\bar{\varepsilon}=0$, a solution exists only when $\mathbb{E} X=\frac{1}{4}\Sigma^{-1}\zeta^{-}$. 
Moving forward, for all $\bar{\varepsilon} \in (0,\bar{\varepsilon}_{c})$ the solution is the unique touching point of the circle and the ellipse and it is obtained by equating the ``equilibrium forces''. However, based on Lemma \ref{lem1}, we may proceed derivative-free (see Appendix \ref{appixA}).
\begin{proposition}[The risk-constrained mmse estimator]\label{prop1} Under the Assumption \ref{as1}, the optimal solution to (\ref{rfp0}) is
\begin{align}
\widehat{ \mathrm{X}}^{\star}_{\lambda^{\star}}{=}\big(I+4\lambda^{*}\Sigma\big)^{-1}\Big(\mathbb{E} X+2\lambda^{*}\big(\mathbb{E}[\Vert X\Vert^2_{2} X]-\mathbb{E}\Vert X\Vert^{2}_{2}\mathbb{E} X\big)\Big), \label{btr}
\end{align}
where $\lambda^{*}\geq {0}$ is the solution to the corresponding dual problem. 
\end{proposition}
\par Because by construction, (\ref{rfp0}) aims to reduce the estimation error variability, it is expected to do so by shifting the optimal estimates towards the areas suffering high loss, that is, towards the tail of the underlying distribution. The next result characterizes the bias of (\ref{btr}): The risk constrained estimator is a shifted version of the conditional expectation by a fraction of the Fisher's moment coefficient of skewness plus the cross third-order statistics of $X$.
\begin{corollary}[Risk-constrained mmse estimator and Fisher's skewness]\label{fish}The estimator (\ref{btr}) operates according to 
\vspace{-5pt}
\begin{align}
\widehat{ \mathrm{X}}^{\star}_{i,u,\lambda}&=\mathbb{E}{ X}_{i,u}{+}\frac{\lambda\sigma_{i}\sqrt{\sigma_{i}}}{1+2\lambda\sigma_{i}}\mathbb{E}\left(\frac{{ X}_{i,u}-\mathbb{E}{ X}_{i,u}}{\sqrt{\sigma_{i}}}\right)^{3} {+}\mathrm{T}_{i,\lambda}, \label{probetter}
\end{align}
where ${ X}_{i,u}$, and $\widehat{ \mathrm{X}}^{\star}_{i,u,\lambda}$ denotes the $i$th component of ${ X}_{u}\equiv U^{\top}{ X}$, and $U^{\top}\widehat{{ \mathrm{X}}}^{*}_{\lambda}$, respectively, $U$ the orthonormal matrix with columns the eigen-vectors of  ${\Sigma}_{{ X}}$, and \\$\small\mathrm{T}_{i,\lambda}=\frac{\lambda}{1+2\lambda\sigma_{i}}\sum_{k\neq{i}}\mathbb{E}{ X}^{2}_{k,u}\mathbb{E}{ X}_{i,u}-\mathbb{E}[{ X}^{2}_{k,u}{ X}_{i,u}].$
\label{coro2}
\end{corollary}
In an effort to minimize the squared error on average while maintaining its variance small, (\ref{btr}) attributes the estimation error made by the conditional expectation to the skewness of the conditional distribution along the eigen-directions of the conditional expected error. In particular, in the one-dimensional case, cross third-order statistics vanish and therefore,  
\vspace{-5pt}
\begin{align}
\widehat{ \mathrm{X}}^{\star}_{\lambda}=\mathbb{E} X+\frac{\lambda\sigma\sqrt{\sigma}}{1+2\lambda\sigma}\mathbb{E}\Bigg( \frac{ X-\mathbb{E}{ X}}{\sqrt{\sigma}}\Bigg)^{3}~,~\lambda\in[0,+\infty). \label{fisherform}
\end{align}
Corollary \ref{coro2} reveals the ``internal mechanism'' with which (\ref{btr}) biases towards the tail of the underline distribution. Per error eigen-direction, (\ref{btr}) compensates for the high losses relative to the expectation, by a fraction of the Fisher's skewness coefficient and a term $\mathrm{T}_{i}$, that captures the cross third-order statistics of the state components. Although (\ref{probetter}) involves the cross third-order statistics of the projected state, these statistics cannot emerge artificially from a coordinate change such as $U^{\top}{ X}$, but are intrinsic to the state itself. Besides, a meaningfull measure of risk should always be coordinate change invariant.
\par Motivated by the geometric picture, and the Lagrangean of (\ref{rfp0}), we note that there exists a unique random element such that (\ref{btr}) is viewed as its corresponding average (see Appendix \ref{proofweightedlemm}). Fortunately, such an element can be expressed as a Radon-Nikodym derivative w.r.t.\ $\mu$.
\begin{lemma}[Risk constrained estimator as a weighted average]\label{weightedlemma} The primal optimal solution of (\ref{rfp0}) can be written as 
\begin{align}
\widehat{ \mathrm{X}}^{\star}_{\lambda}=\mathbb{E}_{ \xi(\widehat{ \mathrm{X}}^{\star}_{\lambda},\lambda)}  X,\label{ex}
\end{align}
\vspace{-5pt}
where 
\begin{align}
\frac{\mathrm{d}{\xi}}{\mathrm{d}\mu}( X)\equiv1+2\lambda\big(\Vert X-\widehat{ \mathrm{X}}^{\star}_{\lambda}\Vert^{2}_{2}-\mathbb{E}\Vert X-\widehat{ \mathrm{X}}^{\star}_{\lambda}\Vert^{2}_{2}\big).    \label{RN}
\end{align}
In particular, for $\lambda\leq{1}/{\gamma^{*}}$, $\gamma^{*}=2\big(\mathbb{E}\Vert\widehat{ \mathrm{X}}^{\star}_{\infty}-\widehat{ \mathrm{X}}^{\star}_{0} \Vert^{2}_{2}+\mathrm{mmse}(\mu)\big)$, it follows that $\frac{\mathrm{d}{\xi}}{\mathrm{d}\mu}\geq{0}$.
\end{lemma}
Combining Lemma \ref{weightedlemma} and (\ref{fisherform}), we obtain \footnote{Where we explicitly refer to the measure w.r.t.\ which expectations are taken.} 
\vspace{-5pt}
\begin{align}
\mathbb{E}_{\xi_{\lambda}} X=\mathbb{E}_{\mu} X+\frac{\lambda\sigma\sqrt{\sigma}}{1+2\lambda\sigma}\mathbb{E}_{\mu}\left(\frac{ X-\mathbb{E}_{\mu} X}{\sqrt{\sigma}}\right)^3, ~\lambda\leq 1/\gamma^{*} \label{erqp}
\end{align}
In words, averaging on the re-distribution $\xi$ is the same as biasing the expectation w.r.t.\ the initial measure $\mu$ towards the tail of $\mu$ based on a fraction of the Fisher's moment coefficient of skewness. This fact can be considered as a strong indication of another manifestation of duality between risk-averse and distributionally robust optimization. 
In addition, for the one dimensional case, 
assuming that $\mu$ is positively skewed, we obtain the bound
\vspace{-5pt}
\begin{align}
\mathbb{E}_{{\xi(\gamma^{*})}} X-\mathbb{E}_{ \mu} X\leq\frac{\sigma\sqrt{\sigma}}{2\Big(\frac{\sigma}{2}{\mathbb{E}\Big|\frac{ X-\mathbb{E} X}{\sqrt{\sigma}}\Big|^{3}+\mathrm{mmse}(\mu)}\Big)+2\sigma}\mathbb{E}\Bigg(\frac{ X-\mathbb{E} X}{\sqrt{\sigma}}\Bigg)^{3}.  \label{trift} 
\end{align}
That is, the expectations of neighboring populations that can be tracked with re-weighting are determined by a fraction of the skewness of the population at hand. In the case of empirical measures, (\ref{trift}) implies that lack of mass in the area of the tail is seen as the presence of risky events.

\section{Distributionally robust mean square error estimation}\label{drosec}
In this section we show that (\ref{ex}) (and therefore the corresponding solution of (\ref{rfp0})) is the solution to a $\chi^{2}-$ DRO problem for all $\chi^{2}-$radii less than a computable constant. We start, by taking into account the equivalence between (\ref{rfp0}) and the mean-deviation risk measure of order $p=2$ (see e.g. \citep{shapiro2021lectures}). To not overload the notation, let us declare $\Vert{ X}-\widehat{ \mathrm{X}}\Vert^{2}_{2}=Z_{\widehat{ \mathrm{X}}}$. Then, (\ref{rfp0}) reads
\begin{equation}
\begin{array}{cl}
\underset{\widehat{ \mathrm{X}}\in\mathbb{R}^{d}}{\operatorname{min}} & \mathbb{E}_{\mu}~Z_{\widehat{ \mathrm{X}}} \\
\text { s.t. } & \mathrm{var}_{\mu}Z_{\widehat{ \mathrm{X}}}\leq \varepsilon,
\end{array} 
\label{mv}
\end{equation}
or equivalently  
\begin{equation}
\begin{array}{cl}
\underset{\widehat{ \mathrm{X}}\in\mathbb{R}^{d}}{\operatorname{min}} & \mathbb{E}_{\mu}~Z_{\widehat{{ X}}} \\
\text { s.t. } & \sqrt{\rule{0pt}{2ex}\mathrm{var}_{\mu}Z_{\widehat{ \mathrm{X}}}}\leq \sqrt{\varepsilon}.
\end{array} 
\label{md}
\end{equation}
It is a standard exercise to verify that the constraint in (\ref{md}) is convex w.r.t. $\widehat{ \mathrm{X}}$ which leads us to the following result.
\begin{lemma}[Relation between the Lagrange multipliers of (\ref{mv}), and (\ref{md})]\label{relationofmultipliers} Let $\lambda^{*}_{mv}$ and $\lambda^{*}_{md}$ be the Lagrange multipliers associated with the constraint of (\ref{mv}), and (\ref{md}) respectively. Then, for all $\varepsilon\geq\varepsilon_{\min}$ it holds
\begin{align}
{\lambda^{*}_{md}(\varepsilon)}=2\lambda^{*}_{mv}(\varepsilon){\sqrt{\varepsilon}}. \label{lmr}
\end{align}
\end{lemma}
Moving forward, strong Lagrangean duality yields
\begin{align}
\begin{array}{cl}
\underset{\widehat{ \mathrm{X}}\in\mathbb{R}^{d}}{\operatorname{min}} & \mathbb{E}_{\mu}~Z_{\widehat{{ X}}} \\
\text { s.t. } & \sqrt{\rule{0pt}{2ex}\mathrm{var}_{\mu}Z_{\widehat{ \mathrm{X}}}}\leq \sqrt{\varepsilon}.
\end{array}
&=-\lambda_{md}^{\star}({\varepsilon})\sqrt{\varepsilon}+\min_{ \widehat{{ X}}\in\mathbb{R}^{d}} \rho(Z_{\widehat{ \mathrm{X}}}),\nonumber
\end{align}
where 
\begin{align}
\rho(Z_{\widehat{ \mathrm{X}}})\equiv\mathbb{E}_{\mu}Z_{\widehat{ \mathrm{X}}}+\lambda_{md}^{\star}({\varepsilon})\sqrt{\mathrm{var}_{\mu}Z_{\widehat{{ X}}}}, \label{primalform}
\end{align}
is the mean-deviation risk measure of order $p=2$ w.r.t.\ $\mu$ (see e.g. \citep{shapiro2021lectures}), with sensitivity constant, the optimal Lagrange multiplier $\lambda^{\star}_{md}(\varepsilon)$.

\subsection{Dual representation of the mean-deviation risk measure of order $p=2$}
Equation (\ref{primalform}) expresses the mean-deviation risk measure in its primal form, where the sensitivity constant is identified with the Lagrange multiplier associated with the constraint in (\ref{rfp0}). Aiming to study robustness in distributional shifts, the place to start is with the dual representation of (\ref{primalform}). Following \citep{shapiro2021lectures}, the variance in (\ref{primalform}), can be expressed via the norm in $\mathcal{L}_{2}$
\begin{align}
(\mathrm{var}_{\mu}Z_{\widehat{ \mathrm{X}}})^{1/2}=(\mathbb{E}_{\mu}|Z_{\widehat{ \mathrm{X}}}-\mathbb{E}_{\mu}Z_{\widehat{ \mathrm{X}}}|^{2})^{1/2}=\Vert Z_{\widehat{ \mathrm{X}}}-\mathbb{E}_{\mu}Z_{\widehat{ \mathrm{X}}}\Vert_{{\mathcal{L}_{2}}}\nonumber,
\end{align}
the latter being identified with its dual norm
$\|Z\|_{\mathcal{L}^{*}_{2}}=\sup _{\|\xi\|_{_{\mathcal{L}_2}} \leq 1}\langle\xi, Z\rangle$, where $||\xi||_{\mathcal{L}_2}\equiv||\xi||_{2}=\sqrt{\mathbb{E}_{\mu}\mathbb{\xi}^{2}}$, and $\langle\xi,Z\rangle=\mathbb{E}_{\mu}\xi{Z}$. Thus, we may write 
$$(\mathrm{var}_{\mu}Z_{\widehat{ \mathrm{X}}})^{1/2}=\sup_{||\xi||_{2}\leq{1}} \langle \xi-\mathbb{E}_{\mu}\xi,Z_{\widehat{ \mathrm{X}}}\rangle,$$
and subsequently (\ref{primalform}) as:
\begin{align}
\hspace{-9pt}\rho(Z_{\widehat{ \mathrm{X}}})&=\langle 1,Z_{\widehat{ \mathrm{X}}}\rangle+\lambda_{md}^{*}({\varepsilon})\sup_{||\xi||_{2}\leq{1}} \langle \xi-\mathbb{E}_{\mu}\xi,Z_{\widehat{ \mathrm{X}}}\rangle\nonumber\\[2pt]
&=\sup_{||\xi||_{2}\leq{1}}\langle Z_{\widehat{ \mathrm{X}}},1\rangle+\lambda_{md}^{*}({\varepsilon}) \langle \xi-\mathbb{E}_{\mu}\xi,Z_{\widehat{ \mathrm{X}}}\rangle\nonumber\\[2pt]
&=\sup_{||\xi||_{2}\leq{1}}\langle 1+\lambda_{md}^{*}({\varepsilon})( \xi-\mathbb{E}_{\mu}\xi),Z_{\widehat{ \mathrm{X}}}\rangle\nonumber\\[2pt]
&=\sup_{\xi^{\prime}\in{\Xi}}\langle\xi^{\prime},Z_{\widehat{ \mathrm{X}}}\rangle,~\Xi\equiv\bigg\{ \xi^{\prime}\in \mathcal{L}^{*}_{2}:\xi^{\prime}=1+\lambda_{md}^{*}({\varepsilon})(\xi-\mathbb{E}_{\mu}\xi),||\xi||_{2}\leq{1}\bigg\}.\label{dualrepresentation}
\end{align}
From (\ref{dualrepresentation}),  $\mathbb{E}_{\mu}\xi^{\prime}=1$, while $\mathbb{E}_{\mu}\xi^{2}\leq{1}$ implies $\mathbb{E}_{\mu}(\xi^{\prime}-1)^2\leq\lambda^{2*}_{md}(\varepsilon)$. Therefore, $\Xi\subseteq \mathrm{U}(\lambda^{*}_{md})$, where 
\begin{align}
\mathrm{U}(\lambda^{*}_{md})\equiv\big\{\xi\in \mathcal{L}^{*}_{2}~:~\mathbb{E}_{\mu}\xi=1,~\mathbb{E}_{\mu}(\xi-1)^{2}\leq\lambda_{md}^{2*}({\varepsilon})\big\}. \label{U}
\end{align}
In addition, let $\xi^{\prime}=\xi-1\Leftrightarrow \xi=1+\xi'+\mathbb{E}\xi'$. Then $\lambda^{2*}_{md}(\varepsilon)\geq \mathbb{E}(\xi-1)^2 = \mathbb{E}\xi^{'2}$, which implies that $\Xi=\mathrm{U}(\lambda^{*}_{md})$.
Thus, we may write
\begin{align}
\rho(Z_{\widehat{ \mathrm{X}}})=\sup_{
\xi\in \mathrm{U}(\lambda^{*}_{md})
}\langle\xi,Z_{\widehat{ \mathrm{X}}}\rangle.\label{50}
\end{align}
Based on the previous discussion, (\ref{rfp0}) is equivalent to 
\begin{align}  
\min_{\widehat{ \mathrm{X}}\in\mathbb{R}^{d}}
\sup_{\xi\in \mathrm{U}(\lambda^{*}_{md})}
\langle\xi,Z_{\widehat{ \mathrm{X}}}\rangle.\label{drs}
\end{align}
The equality constraint in (\ref{U}) can be absorbed and the constraint set may be re-expressed as the chi-square ball of radius $\rho=\lambda^{2\star}_{md}(\varepsilon)\geq{0}$ centered at the probability measure $\mu$ and formed by the $\chi^{2}-$divergence from $\nu$ to $\mu$:
\vspace{-2pt}$$\chi^{2}_{{\lambda^{2\star}_{md}(\varepsilon)}}(\mu)\equiv \big\{\nu \in \mathcal{L}^{\star}_{2}~:~\mathbb{E}_{\mu}(\mathrm{d}\nu/\mathrm{d}\mu-1)^{2}\leq\lambda_{md}^{2\star}({\varepsilon})\big\}.$$
Subsequently then, (\ref{drs}) reads 
\begin{align}
\min_{\widehat{ \mathrm{X}}\in\mathbb{R}^{d}}\left\{ 
\sup_{\nu \in \mathcal{L}^{\star}_{2}}\mathbb{E}_{\nu}Z_{\widehat{ \mathrm{X}}}~:~\nu\in \chi^{2}_{{\lambda^{2\star}_{md}(\varepsilon)}}(\mu)\right\},\label{drop}
\end{align}
Thus the inner optimization problem in (\ref{drop}) is of a linear objective subject to a quadratic constraint over the infinite dimensional (dual) vector space $\mathcal{L}^{\star}_{2}$, and can be easily handled by variational calculus. In particular, by introducing an additional Lagrange multiplier ---for dualizing the quadratic constraint---, it is easy to show that the supremum is attained at 
\begin{align}
\left(\frac{\mathrm{d}\nu}{\mathrm{d}\mathrm{\mu}}\right)^{\star}=\lambda^{\star}_{md}(\varepsilon)\frac{Z_{\widehat{ \mathrm{X}}}-\mathbb{E}_{\mu}Z_{\widehat{ \mathrm{X}}}}{({\mathrm{var}_{{\mu}}Z_{\widehat{ \mathrm{X}}}})^{1/2}}+1.\label{sm}
\end{align}
Note that the objective in (\ref{drop}) is not convex over $\widehat{\mathrm{X}}$ since the expectation is taken w.r.t.\ signed measures. 
At this point, the reader might find it instructive comparing (\ref{sm}) with (\ref{RN}) through (\ref{lmr}), and ask the question: Is there a subset of admissible Lagrange multipliers for which (\ref{drop}) (and therefore (\ref{rfp0})) can be formulated as a distributionally (over probability measures) robust optimization problem? We answer this question with the following result.

\begin{theorem}[Distributionally robust mean square error esimator]\label{Th}
Let $\Lambda\equiv\{\lambda\geq{0}~:~\lambda\leq 1/\gamma^{\star}\},$ be a subset of admissible Lagrange multipliers associated with the constraint of (\ref{rfp0}), with $\gamma^{\star}$ as in Lemma \ref{weightedlemma}, and let $\nu^{\star}$ being determined by $\mu$ as in (\ref{sm}). Then, $(\nu^{\star}(\widehat{\mathrm{X}}^{\star}),\widehat{\mathrm{X}}^{\star}(\nu^{\star}))$ is a saddle point of $f(\xi,\widehat{\mathrm{X}})$ over $\chi^{2}_{{\gamma^{\star}}}(\mu)\times \mathbb{R}^{d}$. In particular, the estimator $\widehat{\mathrm{X}}^{\star}_{\lambda}$ (see Proposition \ref{prop1})  solves the $\chi^{2}-$DRO problem 
\begin{align}
\underset{\widehat{\mathrm{X}}\in\mathbb{R}^{d}}{\mathrm{min}}~\bigg\{ \sup_{0\leq\nu\in \mathcal{L}^{\star}_{2}}~\mathbb{E}_{\sim \nu}\Vert{ X}-\widehat{ \mathrm{X}}\Vert^{2}_{2}~:~\chi^{2}(\nu||\mu)\leq{\lambda}\bigg\},~\lambda\in \Lambda. \label{dro}
\end{align}
\end{theorem}

According to Theorem \ref{Th}, when the risk in (\ref{rfp0}) is left relaxed more than $\varepsilon(1/\gamma^{\star})$\footnote{Here $\varepsilon(1/\gamma^{\star})$ denotes the corresponding to $1/\gamma^{\star}$ level of risk.}, (\ref{dro}) can be viewed as a zero-sum game between the statistician (S) who plays $\widehat{\mathrm{X}}$, and its fictitious opponent (A), who chooses $\nu \in \chi^{2}_{{\gamma^{\star}}}(\mu)$. Then, $\widehat{\mathrm{X}}^{\star}_{\lambda}$ is interpreted as the best response of (S) to the best response $\nu^{\star}$ of (A), with the objective $f$ resting on this equilibrium.
\par Then, according to Corollary \ref{coro2}, the optimal strategy $\widehat{\mathrm{X}}^{\star}$ of (S) is to bias 
-up to a change of co-ordinates- towards the areas of high loss (of $\mu$), with a fraction of the Fisher's moment coefficient of skewness plus some additional cross third-order statistics. The directionality in the play of (S) induces a directionality in the play of (A) as we formally discuss in the next section.

\section{Skewness as the $2-$Wasserstein gradient of the\\ $\chi^{2}-$divergence}\label{wasser}
According to the previous discussion, we observe that the shift of the expectation in (\ref{probetter}), and  (\ref{erqp}) 
attributes directionality to the mass re-distribution play of (A) in Theorem  \ref{Th}. In this subsection we make this observation more formal through the smooth structure of the $2-$Wasserstein space $\mathrm{P}_{2,\text {ac}}(\mathbb{R}^d)$ of probability measures absolutely continuous w.r.t.\ Lebesgue measure with finite second moment, equipped with the $2-$Wasserstein metric $W_2$.
In particular, the dynamic formulation of the Wasserstein distance ~\citep{benamou2000computational} entails the definition of a (pseudo) Riemannian structure ~\citep{chewi2023optimization}. At every point $\rho\in\mathrm{P}_{2,\text{ac}}(\mathbb{R}^{d})$, the tangent space $T_{\rho}\mathrm{P}_{2,\text{ac}}(\mathbb{R}^{d}):=\big\{ \dot{\rho}:\mathbb{R}^{d}\rightarrow\mathbb{R}\big|\int\dot{\rho}=0\big\}$ contains the re-allocation directions and it is parameterized by functions $\omega(x)$ on $\mathbb{R}^{d}$, through the elliptic operator $\omega\mapsto-\nabla\cdot(\rho\nabla\omega)$. The metric
$g:\mathrm{T}\mathrm{P}_{2,\text{ac}}(\mathbb{R}^{d})\times
\mathrm{T}\mathrm{P}_{2,\text{ac}}(\mathbb{R}^{d})\rightarrow C^{\infty}(\mathrm{P}_{2,\text{ac}}(\mathbb{R}^{d}))$ \footnote{ By $C^{\infty}(\mathrm{P}_{2,\text{ac}}(\mathbb{R}^{d}))$ we denote the class of smooth real-valued maps defined on $\mathrm{P}_{2,\text{ac}}(\mathbb{R}^{d})$.
}
has value at the point $\rho\in \mathrm{P}_{2,\text{ac}}(\mathbb{R}^{d})$ ~\citep{otto}:
\begin{align}
g( \dot{\rho}_{1},\dot{\rho}_{2})({\rho}):=\int_{\mathbb{R}^{d}} \langle\nabla \omega_{1}(x), \nabla \omega_{2}(x)\rangle \rho(x) d x=\mathbb{E}_{\rho}\langle\nabla \omega_{1}(x), \nabla \omega_{2}(x)\rangle.\nonumber
\end{align}
Moving forward, let $f\in C^{\infty}(\mathrm{P}_{2,\text{ac}}(\mathbb{R}^{d}))$.  Then, the gradient w.r.t.\ the defined metric, $\mathrm{grad}f:\mathrm{P}_{2,\text{ac}}(\mathbb{R}^{d})\rightarrow \mathrm{T}\mathrm{P}_{2,\text{ac}}(\mathbb{R}^{d})$ is the unique vector field satisfying
\begin{align}
    \mathrm{d}f(X)({\rho})=g(\mathrm{grad}f,X)({\rho}), \nonumber
\end{align}
where $\mathrm{d}f:\mathrm{P}_{2,\text{ac}}(\mathbb{R}^{d})\rightarrow \mathrm{T}^{*}\mathrm{P}_{2,\text{ac}}(\mathbb{R}^{d})$ is the differential of $f$. To derive the gradient $(\mathrm{grad}f)(\rho)$, consider the curves $\rho_{t}:[0,1]\rightarrow\mathrm{P}_{2,\text{ac}}(\mathbb{R}^{d})$ progressing through $\rho$ and realize the tangent vector $\dot{\rho}_{t}=\nabla\cdot(\rho_{t}\nabla\omega(x))$ for some $\omega(x)$. It is 
\begin{align}
\mathrm{d}{f}(\dot{\rho})({\rho})&=\lim_{t\rightarrow{0}}\frac{f(\rho+t\dot{\rho}_{t})-f(\rho)}{t} =\int \frac{\delta{f}}{\delta\rho}(x)\dot{\rho}_{t}dx=-\int \frac{\delta{f}}{\delta\rho}(x) \nabla\cdot(\rho\nabla\omega(x))dx\nonumber\\[5pt]
    &=\int \langle\nabla\frac{\delta{f}}{\delta\rho}(x) ,\nabla\omega(x)\rangle\rho dx.\nonumber
\end{align}
We identify $(\mathrm{grad}f)(\rho)\equiv\nabla~\frac{\delta{f}}{\delta\rho}(x)$, or equivalently $(\mathrm{grad}f)(\rho)=-\nabla\cdot(\rho\nabla~\frac{\delta{f}}{\delta\rho}(x))$, where $\frac{\delta{f}}{\delta\rho}$ denotes the first variation of $f$ w.r.t.\ $\rho$.
For the particular case of the $\chi^{2}-$divergence from $\xi$ to $\mu$, we obtain $
\mathrm{grad}\chi^{2}(\cdot\Vert\mu)(\xi)=2\nabla\frac{\mathrm{d}\xi}{\mathrm{d}\mu},
$
which when evaluated on (\ref{sm}) yields
\begin{align}
\mathrm{grad}\chi^{2}(\cdot\Vert\mu)(\nu_{\lambda^{\star}_{md}})=\frac{2\lambda^{\star}_{md}\nabla\Vert\widehat{ \mathrm{X}}-X\Vert^{2}_{2}}{(\mathrm{var}_{\mu}\Vert\widehat{\mathrm{X}}-X\Vert^{2}_{2})^{1/2}}.   \label{grd}
\end{align}
By plugging $\widehat{\mathrm{X}}=\widehat{ \mathrm{X}}^{\star}_{\lambda_{mv}}$ in (\ref{grd}) and subsequently taking expectations, we obtain (for simplicity in the one dimensional case) 
\begin{align}
\mathbb{E}_{\mu} [\mathrm{grad}\chi^{2}(\cdot\Vert\mu)(\nu_{\lambda^{\star}_{md}})]=\frac{\lambda^{2}\sigma\sqrt{\sigma}}{1+\lambda\sigma}\mathbb{E}\left( \frac{ X-\mathbb{E} X}{\sqrt{\sigma}}\right)^3, \lambda\leq 1/2\gamma^{*}. \label{gradxisquaredbetter}
\end{align}
By combining (\ref{fisherform}) and (\ref{gradxisquaredbetter}) we may further write
\begin{align}
\mathbb{E}_{\nu_{\lambda^{\star}_{md}}} X-\mathbb{E}_{\mu} X  =\frac{\mathbb{E}_{\mu}\mathrm{grad}\chi^{2}(\cdot\Vert\mu)(\nu_{\lambda^{\star}_{md}})}{\lambda},~ \lambda\in (0,1/2\gamma^{*}).\nonumber
\end{align}
Equation (\ref{gradxisquaredbetter}) characterizes infinitesimally the optimal strategy of (A): Given the optimal play of (S), (A) changes the measure in (\ref{drop}) aiming to increase the $\chi^{2}-$divergence from $\nu$ to $\mu$. 
For a more comprehensive review in the theory of optimal transport and applications we refer the reader to \citep{ambrosio2005gradient, villani2009optimal, wibisono2018sampling, figalli2021invitation,  chewi2023optimization}.

\section{Mean squared error of the empirical distributionally robust mse estimator}\label{generalizationsec}
Now that the risk-constrained estimator has all the desired features, we want to study its performance w.r.t.\ the mean squared error. We do so in two steps. First, we consider a simplified version of the empirical distributionally robust estimator and we show that it outperforms the empirical average over a rather large class of models. Second (see Appendix \ref{generalization}), we argue that with high probability this estimator is the empirical version of the distributionally robust estimator (\ref{fisherform}). The next results refer to the one dimensional case ($d=1$), with the extension to other dimensions left open for future exploration.

\par By considering (\ref{fisherform}) w.r.t.\ $\mu_{\mathrm{n}}$, we define the following algorithm
\begin{align}
\widehat{\mathrm{X}}^{\star}_{n,\bar{\lambda}}(D)\equiv\frac{1}{n}\sum_{i} X_{i}+\bar{\lambda}\frac{1}{n}\sum_{i}\Big({ X_{i}-n^{-1}\sum_{i} X_{i}}\Big)^3, \label{simpempiricalbetter}
\end{align}
where $\bar{\lambda}$ is a non-negative and deterministic parameter that replaces the first factor in the second term of (\ref{fisherform}).  
For ease of notation, let us first denote $\bar{\mathrm{ X}}_{n}(D)=n^{-1}\sum_{i} X_{i}$, $\mathrm{T}_{n}(D)=n^{-1}\sum_{i}( X_{i}-\bar{ \mathrm{X}}_{n})^3$. We call (\ref{simpempiricalbetter}) the simplified empirical distributionally robust estimator, which after the above notation reads $\widehat{\mathrm{X}}^{\star}_{n,\bar{\lambda}}=\bar{\mathrm{ X}}_{n}+\bar{\lambda}\mathrm{T}_{n}$, and achieves mean squared error 
\begin{align}
        \mathbb{E}(\widehat{\mathrm{X}}^{\star}_{n,\bar{\lambda}}- X)^2
        &\leq \mathbb{E}(\bar{\mathrm{X}}_{n}- X)^2+\bar{\lambda}^2(\mathbb{E}\mathrm{T}^{2}_{n}+\epsilon)-2\bar{\lambda}\mathbb{E}\big[( X-\bar{ \mathrm{X}}_{n})\mathrm{T}_{n}\big], \label{bounddomination}
\end{align}
for any $\epsilon>0$. Strict domination follows for all those positive $\bar{\lambda}$ such that \vspace{-3pt}
\begin{align}
\bar{\lambda}<{2\mathbb{E}\big[( X- \bar{\mathrm{X}}_{n})\mathrm{T}_{n}\big]}/({\mathbb{E}\mathrm{T}^{2}_{n}+\epsilon}), \label{in} 
\end{align}
only under reasonable conditions that render the right-hand side of (\ref{in}) strictly positive. 
Let $Q$ be the set of all probability measures with finite second order moment and $g:Q\rightarrow \mathbb{R}$, the real-valued map with $g(\mu)\equiv 3m^{2}_{2}-m_{4}$, where $m_{2}$ and $m_{4}$ is the second and fourth central moment of $\mu$, respectively. Note that the Gaussian measures are all in $g^{-1}(0)$, and define the set
$\mathrm{N}\equiv g^{-1}((0,+\infty))\subset Q$.  Strict domination of (\ref{simpempiricalbetter}) is demonstrated by the following 
\begin{theorem}[Strict domination] \label{theorem}Let $\widehat{ \mathrm{X}}^{\star}_{n,\bar{\lambda}}$ be as in (\ref{simpempiricalbetter}), and $\bar{\mathrm{ X}}_{n}$ be the empirical average. Fix an $\epsilon>0$. Then, for probability measures $\mu\in \mathrm{N}$, and for any $n\geq 3$,
\begin{align}
\mathbb{E}|\widehat{ \mathrm{X}}^{\star}_{n,\bar{\lambda}}- \mathbb{E}X|^2<\mathbb{E}|\bar{\mathrm{X}}_{n}- \mathbb{E}X|^2, \label{gn}
\end{align}
for all $\bar{\lambda}< \frac{1}{n}\left(\frac{8g(\mu)}{\mathbb{E}\mathrm{T}^{2}_{n}+\epsilon}\right)$.
\end{theorem}
As mentioned earlier, Gaussian measures are among those with $3m^{2}_{2}=m_{4}$. In this case $\bar{\lambda}=0$, the condition (\ref{bounddomination}) is satisfied with equality for any $n \geq {1}$, and the mse performance of (\ref{simpempiricalbetter}) reduces to that of the empirical average. On the contrary, when the data are generated from $\mu\in \mathrm{N}$, (\ref{simpempiricalbetter}) outperforms the average as soon as the parameter $\bar{\lambda}$ is chosen inside the prescribed interval. 
Interestingly, domination remains feasible for arbitrarily large data sets as soon as $\bar{\lambda}$ is sent to zero with a rate of at least $O(1/n)$. The distributionally robust estimator achieves better mse performance compared to the sample mean by leveraging its bias term. Further, if $\bar{\lambda}\mapsto \varphi(\bar{\lambda})\equiv \mathbb{E}|\widehat{ \mathrm{X}}^{\star}_{n,\bar{\lambda}}- X|^2$, then the slope of $\varphi$ at zero 
\begin{align}
\varphi^{\prime}(0)=-(3m_{2}^2-m_{4})\frac{n^{2}-3n+2}{n^{3}}\label{slope} 
\end{align}
It is $\varphi^{\prime}(0)\leq{0}$ for all $\mu\in \mathrm{N}\cup g^{-1}(0)$, and the maximum value is attained for $\mu \in g^{-1}(0)$. In that case of course the best choice is $\bar{\lambda}=0$. On the contrary, for all those probability measures $\mu\in \mathrm{N}$ strict domination follows for (see also Appendix \ref{generalization}):
\begin{align}
\bar{\lambda}< \frac{2(3m^{2}_{2}-m_{4})}{\mathbb{E}\mathrm{T}^{2}_{n}+\epsilon}\frac{n^{2}-3n+2}{n^3}=\frac{-\varphi^{\prime}(0)}{\mathbb{E}\mathrm{T}^{2}_{n}+\epsilon}.  
\end{align}
The more negative the slope at zero is, the wider the feasible set becomes. A very negative slope at zero permits large improvement in the mse performance even for very small values of $\bar{\lambda}$ thus rendering, at the same time, the estimator as unbiased as possible. Since the slope at zero is determined by $-(3m^{2}_{2}-m_{4})$, those models with large $|g(\mu)|=|3m^{2}_{2}-m_{4}|$ trade more favorably mse performance for unbiasedness. That said, for models with large value of $g$, values of $\bar{\lambda}$ close to zero achieve a non-trivial performance difference even for large number of samples. However, on the basis of (\ref{slope}), large data sets tend to flatten out $\varphi$ at zero, and to maintain performance $\bar{\lambda}$ has to be increased to 
\begin{align}
    \bar{\lambda}^{\star}_{n}=\mathrm{argmin}\{\varphi(\bar{\lambda}),~\bar{\lambda}\leq -\varphi^{\prime}(0)/({\mathbb{E}\mathrm{T}^{2}_{n}+\varepsilon})\},
\end{align}
thus rendering $\widehat{\mathrm{X}}^{\star}_{n,\bar{\lambda}^{\star}_{n}}$ more biased. Lastly, $\bar{\lambda}^{\star}_{n}$  approaches zero as the feasible set shrinks with a rate of at least $O(1/n)$ pushing  $\widehat{\mathrm{X}}^{\star}_{n,\bar{\lambda}^{\star}_{n}}$ to the empirical average.

\section{Conclusion and future work}
We deployed the main insights of DRO to come up with a distributionally robust estimator; its empirical version serving as a better estimator for the mean compared to the empirical average w.r.t.\ the mean squared error for all platykurtic probability distributions. The aforesaid estimator biases-up to a change of coordinates-towards the tail of the distribution with the Fisher's moment coefficient of skewness plus additional cross third order statistics and on top of that, it can be written as an expectation w.r.t.\ an optimal re-weighting of the original measure. This optimal re-weighting along with the optimal estimator form a saddle point of the mse cost and the bias characterizes the directionality of the infinitesimal play through the Wasserstein geometry of the space of measures. Lastly, extending Theorem \ref{theorem} in the multi-dimensional setting is left for future investigation.

\newpage

\bibliographystyle{plainnat} 
\bibliography{refs}

\begin{thebibliography}{36}
\providecommand{\natexlab}[1]{#1}
\providecommand{\url}[1]{\texttt{#1}}
\expandafter\ifx\csname urlstyle\endcsname\relax
  \providecommand{\doi}[1]{doi: #1}\else
  \providecommand{\doi}{doi: \begingroup \urlstyle{rm}\Url}\fi

\bibitem[Abeille et~al.(2016)Abeille, Lazaric, Brokmann, et~al.]{abeille2016lqg}
Marc Abeille, Alessandro Lazaric, Xavier Brokmann, et~al.
\newblock Lqg for portfolio optimization.
\newblock \emph{arXiv preprint arXiv:1611.00997}, 2016.

\bibitem[Ambrosio et~al.(2005)Ambrosio, Gigli, and Savar{\'e}]{ambrosio2005gradient}
Luigi Ambrosio, Nicola Gigli, and Giuseppe Savar{\'e}.
\newblock \emph{Gradient flows: in metric spaces and in the space of probability measures}.
\newblock Springer Science \& Business Media, 2005.

\bibitem[Bayraksan and Love(2015)]{bayraksan2015data}
G{\"u}zin Bayraksan and David~K Love.
\newblock Data-driven stochastic programming using phi-divergences.
\newblock In \emph{The operations research revolution}, pages 1--19. INFORMS, 2015.

\bibitem[Benamou and Brenier(2000)]{benamou2000computational}
Jean-David Benamou and Yann Brenier.
\newblock A computational fluid mechanics solution to the {M}onge-{K}antorovich mass transfer problem.
\newblock \emph{Numerische Mathematik}, 84\penalty0 (3):\penalty0 375--393, 2000.

\bibitem[Bertsimas and Sim(2004)]{bertsimas2004price}
Dimitris Bertsimas and Melvyn Sim.
\newblock The price of robustness.
\newblock \emph{Operations research}, 52\penalty0 (1):\penalty0 35--53, 2004.

\bibitem[Blanchet and Shapiro(2023)]{blanchet2023statistical}
Jose Blanchet and Alexander Shapiro.
\newblock Statistical limit theorems in distributionally robust optimization.
\newblock In \emph{2023 Winter Simulation Conference (WSC)}, pages 31--45. IEEE, 2023.

\bibitem[Blanchet et~al.(2021)Blanchet, Murthy, and Nguyen]{blanchet2021statistical}
Jose Blanchet, Karthyek Murthy, and Viet~Anh Nguyen.
\newblock Statistical analysis of wasserstein distributionally robust estimators.
\newblock In \emph{Tutorials in Operations Research: Emerging Optimization Methods and Modeling Techniques with Applications}, pages 227--254. INFORMS, 2021.

\bibitem[Blanchet et~al.(2023)Blanchet, Kuhn, Li, and Taskesen]{blanchet2023unifying}
Jose Blanchet, Daniel Kuhn, Jiajin Li, and Bahar Taskesen.
\newblock Unifying distributionally robust optimization via optimal transport theory.
\newblock \emph{arXiv preprint arXiv:2308.05414}, 2023.

\bibitem[Blanchet et~al.(2024)Blanchet, Li, Lin, and Zhang]{blanchet2024distributionally}
Jose Blanchet, Jiajin Li, Sirui Lin, and Xuhui Zhang.
\newblock Distributionally robust optimization and robust statistics.
\newblock \emph{arXiv preprint arXiv:2401.14655}, 2024.

\bibitem[Chewi(2023)]{chewi2023optimization}
Sinho Chewi.
\newblock \emph{An optimization perspective on log-concave sampling and beyond}.
\newblock PhD thesis, Massachusetts Institute of Technology, 2023.

\bibitem[Devroye et~al.(2003)Devroye, Sch{\"a}fer, Gy{\"o}rfi, and Walk]{devroye2003estimation}
Luc Devroye, Dominik Sch{\"a}fer, L{\'a}szl{\'o} Gy{\"o}rfi, and Harro Walk.
\newblock The estimation problem of minimum mean squared error.
\newblock \emph{Statistics \& Decisions}, 21\penalty0 (1):\penalty0 15--28, 2003.

\bibitem[Duchi and Namkoong(2018)]{duchi2018learning}
John Duchi and Hongseok Namkoong.
\newblock Learning models with uniform performance via distributionally robust optimization.
\newblock \emph{arXiv preprint arXiv:1810.08750}, 2018.

\bibitem[Duchi and Namkoong(2019)]{duchi2019variance}
John Duchi and Hongseok Namkoong.
\newblock Variance-based regularization with convex objectives.
\newblock \emph{Journal of Machine Learning Research}, 20\penalty0 (68):\penalty0 1--55, 2019.

\bibitem[Esfahani and Kuhn(2015)]{esfahani2015data}
Peyman~Mohajerin Esfahani and Daniel Kuhn.
\newblock Data-driven distributionally robust optimization using the wasserstein metric: Performance guarantees and tractable reformulations.
\newblock \emph{arXiv preprint arXiv:1505.05116}, 2015.

\bibitem[Figalli and Glaudo(2021)]{figalli2021invitation}
Alessio Figalli and Federico Glaudo.
\newblock \emph{An invitation to optimal transport, Wasserstein distances, and gradient flows}.
\newblock 2021.

\bibitem[Gao and Kleywegt(2023)]{gao2023distributionally}
Rui Gao and Anton Kleywegt.
\newblock Distributionally robust stochastic optimization with wasserstein distance.
\newblock \emph{Mathematics of Operations Research}, 48\penalty0 (2):\penalty0 603--655, 2023.

\bibitem[Girshick and Savage(1951)]{girshick1951bayes}
MA~Girshick and LJ~Savage.
\newblock Bayes and minimax estimates for quadratic loss functions.
\newblock In \emph{Proceedings of the Second Berkeley Symposium on Mathematical Statistics and Probability}, volume~2, pages 53--74. University of California Press, 1951.

\bibitem[Gotoh et~al.(2018)Gotoh, Kim, and Lim]{gotoh2018robust}
Jun-ya Gotoh, Michael~Jong Kim, and Andrew~EB Lim.
\newblock Robust empirical optimization is almost the same as mean--variance optimization.
\newblock \emph{Operations research letters}, 46\penalty0 (4):\penalty0 448--452, 2018.

\bibitem[James and Stein(1961)]{james1961estimation}
William James and Charles Stein.
\newblock Estimation with quadratic loss.
\newblock In \emph{Breakthroughs in statistics: Foundations and basic theory}, pages 443--460. Springer, 1961.

\bibitem[Kalogerias et~al.(2019)Kalogerias, Chamon, Pappas, and Ribeiro]{kalogerias2019risk}
Dionysios~S Kalogerias, Luiz~FO Chamon, George~J Pappas, and Alejandro Ribeiro.
\newblock Risk-aware mmse estimation.
\newblock \emph{arXiv preprint arXiv:1912.02933}, 2019.

\bibitem[Kalogerias et~al.(2020)Kalogerias, Chamon, Pappas, and Ribeiro]{kalogerias2020better}
Dionysios~S Kalogerias, Luiz~FO Chamon, George~J Pappas, and Alejandro Ribeiro.
\newblock Better safe than sorry: Risk-aware nonlinear bayesian estimation.
\newblock In \emph{ICASSP 2020-2020 IEEE international conference on acoustics, speech and signal processing (ICASSP)}, pages 5480--5484. IEEE, 2020.

\bibitem[Kuhn et~al.(2019)Kuhn, Esfahani, Nguyen, and Shafieezadeh-Abadeh]{kuhn2019wasserstein}
Daniel Kuhn, Peyman~Mohajerin Esfahani, Viet~Anh Nguyen, and Soroosh Shafieezadeh-Abadeh.
\newblock Wasserstein distributionally robust optimization: Theory and applications in machine learning.
\newblock In \emph{Operations research \& management science in the age of analytics}, pages 130--166. Informs, 2019.

\bibitem[Lam(2016)]{lam2016robust}
Henry Lam.
\newblock Robust sensitivity analysis for stochastic systems.
\newblock \emph{Mathematics of Operations Research}, 41\penalty0 (4):\penalty0 1248--1275, 2016.

\bibitem[Lee and Raginsky(2018)]{lee2018minimax}
Jaeho Lee and Maxim Raginsky.
\newblock Minimax statistical learning with wasserstein distances.
\newblock \emph{Advances in Neural Information Processing Systems}, 31, 2018.

\bibitem[Markowitz(1952)]{e5a1bb8f-41b7-35c6-95cd-8b366d3e99bc}
Harry Markowitz.
\newblock Portfolio selection.
\newblock \emph{The Journal of Finance}, 7\penalty0 (1):\penalty0 77--91, 1952.
\newblock ISSN 00221082, 15406261.
\newblock URL \url{http://www.jstor.org/stable/2975974}.

\bibitem[Namkoong and Duchi(2016)]{namkoong2016stochastic}
Hongseok Namkoong and John~C Duchi.
\newblock Stochastic gradient methods for distributionally robust optimization with f-divergences.
\newblock \emph{Advances in neural information processing systems}, 29, 2016.

\bibitem[Nguyen et~al.(2023)Nguyen, Shafieezadeh-Abadeh, Kuhn, and Mohajerin~Esfahani]{nguyen2023bridging}
Viet~Anh Nguyen, Soroosh Shafieezadeh-Abadeh, Daniel Kuhn, and Peyman Mohajerin~Esfahani.
\newblock Bridging bayesian and minimax mean square error estimation via wasserstein distributionally robust optimization.
\newblock \emph{Mathematics of Operations Research}, 48\penalty0 (1):\penalty0 1--37, 2023.

\bibitem[Otto(2001)]{otto}
Felix Otto.
\newblock The geometry of dissipative evolution equations: The porous medium equation.
\newblock \emph{Communications in Partial Differential Equations}, 26\penalty0 (1-2):\penalty0 101--174, 2001.

\bibitem[Scarf et~al.(1957)Scarf, Arrow, and Karlin]{scarf1957min}
Herbert~E Scarf, KJ~Arrow, and S~Karlin.
\newblock \emph{A min-max solution of an inventory problem}.
\newblock Rand Corporation Santa Monica, 1957.

\bibitem[Shapiro et~al.(2021)Shapiro, Dentcheva, and Ruszczynski]{shapiro2021lectures}
Alexander Shapiro, Darinka Dentcheva, and Andrzej Ruszczynski.
\newblock \emph{Lectures on stochastic programming: modeling and theory}.
\newblock SIAM, 2021.

\bibitem[Staib and Jegelka(2019)]{staib2019distributionally}
Matthew Staib and Stefanie Jegelka.
\newblock Distributionally robust optimization and generalization in kernel methods.
\newblock \emph{Advances in Neural Information Processing Systems}, 32, 2019.

\bibitem[Van~Parys et~al.(2021)Van~Parys, Esfahani, and Kuhn]{van2021data}
Bart~PG Van~Parys, Peyman~Mohajerin Esfahani, and Daniel Kuhn.
\newblock From data to decisions: Distributionally robust optimization is optimal.
\newblock \emph{Management Science}, 67\penalty0 (6):\penalty0 3387--3402, 2021.

\bibitem[Verhaegen and Verdult(2007)]{verhaegen2007filtering}
Michel Verhaegen and Vincent Verdult.
\newblock \emph{Filtering and system identification: a least squares approach}.
\newblock Cambridge university press, 2007.

\bibitem[Villani et~al.(2009)]{villani2009optimal}
C{\'e}dric Villani et~al.
\newblock \emph{Optimal transport: old and new}, volume 338.
\newblock Springer, 2009.

\bibitem[Wang and Bovik(2009)]{wang2009mean}
Zhou Wang and Alan~C Bovik.
\newblock Mean squared error: Love it or leave it? a new look at signal fidelity measures.
\newblock \emph{IEEE signal processing magazine}, 26\penalty0 (1):\penalty0 98--117, 2009.

\bibitem[Wibisono(2018)]{wibisono2018sampling}
Andre Wibisono.
\newblock Sampling as optimization in the space of measures: The langevin dynamics as a composite optimization problem.
\newblock In \emph{Conference on Learning Theory}, pages 2093--3027. PMLR, 2018.

\end{thebibliography}
\appendix

\section{Appendix A (Proofs of Section \ref{sec2})}\label{appixA}
\begin{proof}[Proof of Lemma \ref{lem1}]
The square of the constraint of (\ref{rfp0}) is written as\footnote{All expectations are taken w.r.t.\ $\mu$.}:
\begin{align}
\Big(\|{ X}-\widehat{\mathrm{ X}}\|^{2}_{2}-\mathbb{E} 
||{ X}-\widehat{\mathrm{ X}}||^{2}_{2}
\Big)^{2}
  &=(\|{ X}\|^{2}_{2}+\mathbb{E} \|{ X}\|^{2}_{2})^{2}+4\widehat{ \mathrm{X}}^{\top}({ X}-\mathbb{E} X)({ X}-\mathbb{E} X)^{\top}\widehat{ \mathrm{X}}\nonumber\\[2pt]&\hspace{30pt}+4(|| X||^{2}_{2}
  -\mathbb{E}\{|| X||^{2}_{2}\})(\mathbb{E} X-{ X})^{\top}\widehat{ \mathrm{X}}
\nonumber\\[2pt]&=
\left[\begin{array}{cc}
1~ &~ \widehat{\mathrm{ X}}^{\top}
\end{array}\right]
\left[\begin{array}{ll}
\alpha({ X}) & \zeta({ X})^{\top} \\
\zeta({ X}) & \gamma({ X})
\end{array}\right]
\left[~\begin{array}{c}
1 \\ \widehat{\mathrm{ X}}
~\end{array}\right], \label{eft}
\end{align}
where 
\begin{align}
\alpha({ X})&=(||{ X}||^{2}_{2}-\mathbb{E}\| X\|^{2}_{2})^{2}  \nonumber   \\[2pt]    
\zeta({ X})&=2(\|{ X}\|_2^2-\mathbb{E}\{\|{ X}\|_2^2 \})(\mathbb{E}{ X} -{ X})\nonumber
\\[2pt]
\gamma({ X})&=4({ X}-\mathbb{E} X)( X-\mathbb{E} X)^{\top}.\nonumber
\end{align}
Thus, by taking expectations
\begin{align}
&\mathbb{E}
\Big(||{ X}-\widehat{\mathrm{ X}}||^{2}_{2}-\mathbb{E}_{\mu} 
||{ X}-\widehat{\mathrm{ X}}||^{2}_{2}
\Big)^{2}=
\left[\begin{array}{cc}
1~ &~ \widehat{\mathrm{ X}}^{\top}
\end{array}\right]
\left[\begin{array}{ll}
\alpha & \zeta^{\top} \\
\zeta& \Gamma
\end{array}\right]
\left[~\begin{array}{c}
1 \\ \widehat{\mathrm{ X}}
~\end{array}\right],\label{oxt}
\end{align}
where $\alpha=\mathbb{E} \alpha({ X}) $, $\zeta=\mathbb{E}\zeta({ X})$, and $\Gamma=4\Sigma$.
Equation (\ref{oxt}) rests on the permutation-invariance property of the $\mathrm{trace}$.
Without loss of generality, let us assume that the covariance matrix $\Sigma$ is strictly positive, and consider its Schur complement w.r.t.\ $\mathrm{\Gamma}$ thus obtaining the following standard factorization~(see e.g.~\citep[Lemma 2.3]{verhaegen2007filtering}):
\begin{align}
&\mathbb{E}
\big(\Vert{ X}-\widehat{\mathrm{ X}}\Vert^{2}_{2}-\mathbb{E} 
\Vert{ X}-\widehat{\mathrm{ X}}\Vert^{2}_{2}
\big)^{2}=\left[\begin{array}{c}
1\\~\Gamma^{-1}\zeta-\widehat{ \mathrm{X}}
\end{array}\right]^{\top}\left[\begin{array}{ll}
\alpha-\zeta^{\top}\Gamma^{-1}\zeta & ~{0}\\0& \Gamma
\end{array}\right]\left[\begin{array}{c}
1\\~\Gamma^{-1}\zeta-\widehat{ \mathrm{X}}
\end{array}\right]_{\boldsymbol{.}}
\label{enea}
\end{align}
\end{proof}

\begin{proof}[Proof of Proposition \ref{prop1}]
From Lemma \ref{lem1}, problem (\ref{rfp0}) can equivalently be written as:
\begin{equation}
\begin{array}{rl}
\underset{\widehat{\mathrm{ X}}\in\mathbb{R}^{d}}{\mathrm{min}} & \Vert\widehat{\mathrm{ X}}-\mathbb{E} X\Vert^{2}_{2}\\
\mathrm{s.t.} & \Vert\widehat{\mathrm{ X}}-\frac{1}{4}{\Sigma}^{-1}\zeta^{-}\Vert^{2}_{4{\Sigma}}\leq\bar{\varepsilon}
\end{array},~ \bar{\varepsilon}\geq0~=~ \underset{{\widehat{\mathrm{ X}}\in\mathbb{R}^{d}}}{\mathrm{min}}\sup_{\lambda \geq {0}}\mathrm{L}(\widehat{\mathrm{ X}},\lambda),\nonumber
\end{equation}
where $\mathrm{L}:\mathbb{R}^{d}\times\mathbb{R}^{+}\rightarrow\mathbb{R}$ is the Lagrangean of (\ref{rfp3}) given by
\begin{align}
    \mathrm{L}(\widehat{\mathrm{ X}},\lambda)=\Vert\widehat{\mathrm{ X}}-\mathbb{E} X\Vert^{2}_{2}+
\lambda~\Vert\widehat{\mathrm{ X}}-\frac{1}{4}{\Sigma}^{-1}\zeta^{-}\Vert^{2}_{4{\Sigma}}-\lambda\bar{\varepsilon},\label{lag}
\end{align}
and $\lambda\in\mathbb{R}^{+}$ is the multiplier associated with the constraint of (\ref{rfp3}). Slater's condition is directly verified from (\ref{rfp3}) and therefore,  due to convexity, the dual-optimal value is attained, and (\ref{rfp3}) has zero duality gap:
\begin{align}
\underset{\widehat{\mathrm{ X}}\in\mathbb{R}^{d}}{\mathrm{min}}\sup_{\lambda \geq {0}}\mathrm{L}(\widehat{\mathrm{ X}},\lambda)=\sup_{\lambda \geq {0}}\underset{{\widehat{\mathrm{ X}}\in\mathbb{R}^{d}}}{\mathrm{min}}\mathrm{L}(\widehat{\mathrm{ X}},\lambda).\nonumber
\end{align}
The Lagrangian reads 
\begin{align}
&\mathrm{L}(\widehat{\mathrm{ X}},\lambda)=\left[\begin{array}{cc}
\widehat{\mathrm{ X}}^{\top}~ &~ 1
\end{array}\right]
\hspace{-4pt}\left[\begin{array}{ll}
{I}+4\lambda \Sigma~~&~~ -\big(\mathbb{E} X+\lambda\zeta^{-}\big) \\
-\big(\mathbb{E} X+\lambda\zeta^{-} \big)^{\top}&  ||\mathbb{E} X||^{2}_{2}+\lambda ||\frac{1}{4}{\Sigma}^{-1}\zeta^{-}||^{2}_{2}
\end{array}\right]
\hspace{-4pt}\left[~\begin{array}{c}
\widehat{\mathrm{ X}} \\ 1
~\end{array}\right]-\lambda\bar{\varepsilon},\nonumber
\end{align}
and by performing a similar decomposition as in (\ref{enea}) we obtain
\begin{align}
\mathrm{L}(\widehat{\mathrm{ X}},\lambda)&= \Vert \widehat{\mathrm{ X}}-({I}+4\lambda {\Sigma})^{-1}(\mathbb{E}  X+\lambda{\zeta}^{-})\Vert^{2}_{I+4\lambda{\Sigma}}\nonumber\\[2pt]
&-\big(\mathbb{E}X+\lambda\zeta^{-} \big)^{\top}\big( {I}+4\lambda\Sigma\big)^{-1}\big(\mathbb{E} X+\lambda\zeta^{-} \big)\nonumber\\[2pt]
&+\Vert\mathbb{E} X\Vert^{2}_{2}+\lambda \Vert\frac{1}{4}{\Sigma}^{-1}{\zeta}^{-}\Vert^{2}_{{\Sigma}}-\lambda\bar{\varepsilon}. \label{lag}
\end{align}
The last two terms in (\ref{lag}) do not depend on $\widehat{\mathrm{ X}}$, and thus
\begin{align}
\widehat{\mathrm{ X}}^{\star}_{\lambda}=({I}+4\lambda {\Sigma})^{-1}(\mathbb{E}  X+\lambda{\zeta}^{-}), \label{better0}
\end{align} 
is primal-optimal.
\end{proof}
\begin{proof}[Proof of Corollary \ref{fish}]
Let us denote $\Pi(\lambda)=({I}+4\lambda {\Sigma})^{-1}$. Differentiation of (\ref{better0}) w.r.t. $\lambda\in(0,+\infty)$, gives
\begin{align}
\mathrm{d}_{\lambda}\widehat{\mathrm{ X}}^{\star}_{\lambda}
&=-4\Pi(\lambda)\Sigma\Pi(\lambda)(\mathbb{E}\{  X\}+\lambda{\zeta}^{-})+\Pi(\lambda){\zeta}^{-}, \label{dgs}
\end{align}
where the commutator $[\Pi,\Sigma]\equiv \Pi\Sigma-\Sigma_{ X\mid{X}}\Pi=0$, and therefore we may write
\begin{align}
\mathrm{d}_{\lambda}\widehat{\mathrm{ X}}^{\star}_{\lambda}&= \Pi(\lambda)^{2}\Big(-4\Sigma(\mathbb{E}  X+\lambda{\zeta}^{-})+\Pi(\lambda)^{-1}{\zeta}^{-}\Big)\nonumber\\[2pt]
&=-4\Pi(\lambda)^{2}\Sigma\Big(\mathbb{E} X-(4\Sigma)^{-1}\zeta^{-}\Big).\label{frog}
\end{align}
It is worth noticing that the dependency on $\lambda$ has transferred in $\Pi$, and that $\mathbb{E} X=\widehat{ \mathrm{X}}^{\star}_{0}$, and $(4\Sigma)^{-1}\zeta^{-}=\widehat{ \mathrm{X}}^{\star}_{\infty}$. Thus, we declare $\widehat{\Delta X}^{\star}_{0\infty}=(4\Sigma)^{-1}\zeta^{-}-\mathbb{E} X$.  By integrating (\ref{frog}) w.r.t. $\lambda$ we obtain 
\begin{align}
\widehat{\mathrm{ X}}^{\star}_{\lambda}&=\widehat{\mathrm{ X}}^{\star}_{0}+U\mathrm{H}(\lambda)U^{\top}\widehat{\Delta X}^{\star}_{0\infty}.\label{rbetter}
\end{align}
where 
$\mathrm{H}(\lambda)=\mathrm{diag}(\frac{4\lambda\sigma}{1+4\lambda\sigma})$, and $\Sigma=U\Lambda U^{\top}$ refers to the spectral decomposition of $\Sigma$.
That is, the optimal estimates are shifted versions of the conditional expectation by a transformed version of the vector $\widehat{\Delta X}^{\star}_{0\infty}$. Recall that this vector encodes the difference between the center of the circle and the center of the ellipsoid. 
From (\ref{rbetter}) we obtain
\begin{align}
U^{\top}\widehat{\mathrm{ X}}^{\star}_{\lambda}=U^{\top}\widehat{\mathrm{ X}}^{\star}_{0}+\mathrm{H}(\lambda)\big(U^{\top}\widehat{\mathrm{ X}}^{\star}_{0}-U^{\top}\widehat{\mathrm{ X}}^{\star}_{\infty}\big)
\end{align}
Further,  
\begin{align}
U^{\top}\widehat{\mathrm{ X}}^{\star}_{0}&=U^{\top}\mathbb{E}{ X} =\mathbb{E}{ X}_{u},
\end{align}where ${ X}_{u}\equiv U^{\top}{ X}$.
In addition, since $$U^{\top}\mathbb{E}[\| { X}\|^{2}_{2}{ X}]=\mathbb{E}[\|UU^{\top} { X}\|^{2}_{2}U^{\top}{ X}],$$and $U^{\top}$ is unitary,
\begin{align}
U^{\top}\widehat{\mathrm{ X}}^{\star}_{\infty}&=\frac{1}{2}\Lambda^{-1}\bigg(
\mathbb{E}[\| { X}_{u}\|^{2}_{2}{ X}_{u}]-\mathbb{E} \|{ X}_{u}\|^{2}_{2}\mathbb{E}{ X}_{u} \bigg)
\end{align}
By $\widehat{\mathrm{ X}}^{\star}_{i,u,\lambda}$, and ${ X}_{i,u}$ we mean the $i$th component of $\widehat{\mathrm{ X}}^{\star}_{u,\lambda}\equiv U^{\top}\widehat{\mathrm{ X}}^{\star}_{\lambda}$, and ${ X}_{u}$, respectively. Also, note that since the singular values of the covariance matrix are coordinate-change invariant, $\Lambda=\Lambda_{{ X}_{u}}$, or $\mathrm{var}[{ X}_{i}]=\mathrm{var}[{ X}_{i,u}]\equiv\sigma_{i}$. Thus,
\begin{align}
\widehat{\mathrm{ X}}^{\star}_{i,u,\lambda}&=\mathbb{E}{ X}_{i,u}+\frac{2\lambda\sigma_{i}}{1+2\lambda\sigma_{i}}\Bigg(\mathbb{E} { X}_{i,u}-\frac{1}{2\sigma_{i}}\Big(\sum_{k}\mathbb{E}[{ X}^{2}_{k,u}{ X}_{i,u}]
-\sum_{k}\mathbb{E}{ X}^{2}_{k,u} \mathbb{E}{ X}_{i,u}\Big)
\Bigg)\nonumber\\[2pt]
&=\mathbb{E}{ X}_{i,u}
\nonumber\\[2pt]
&+\frac{2\lambda\sigma_{i}}{1+2\lambda\sigma_{i}}\Bigg(\mathbb{E} { X}_{i,u}-\frac{1}{2\sigma_{i}}\Big(\mathbb{E}{ X}^{3}_{i,u}
-\mathbb{E}{ X}^{2}_{i,u} \mathbb{E}{ X}_{i,u}\Big)
\Bigg)\nonumber\\[2pt]
&+\frac{2\lambda\sigma_{i}}{1+2\lambda\sigma_{i}}\Bigg(-\frac{1}{2\sigma_{i}}\Big(\sum_{k\neq{i}}\mathbb{E}[{ X}^{2}_{k,u}{ X}_{i,u}]
-\sum_{k\neq{i}}\mathbb{E}{ X}^{2}_{k,u} \mathbb{E}{ X}_{i,u}\Big)
\Bigg) \label{de}
\end{align}
Completing the third power in the second term of (\ref{de}), and subsequently setting the last term equal to $\mathrm{T}_{i,\lambda}$ concludes the proof.
\end{proof}
\begin{proof}[Proof of Lemma \ref{weightedlemma}]
\label{proofweightedlemm}
The optimallity condition for the corresponding to (\ref{rfp0}) Lagrangean yields:
\begin{align}
\mathrm{d}_{\widehat{ \mathrm{X}}}\mathrm{L}(\widehat{ \mathrm{X}}^{\star},\lambda)=0,\nonumber
\end{align}
or
\begin{align}
2\mathbb{E}(\widehat{ \mathrm{X}}^{\star}_{\lambda}- X)+\lambda\mathbb{E}\bigg\{2\Big(\Vert X-\widehat{ \mathrm{X}}^{\star}_{\lambda}\Vert^{2}_{2}-\mathbb{E}\Vert X-\widehat{ \mathrm{X}}^{\star}_{\lambda}\Vert^{2}_{2} \Big)\big( 2( X-\widehat{ \mathrm{X}}^{\star}_{\lambda})-2\mathbb{E}( X-\widehat{ \mathrm{X}}^{\star}_{\lambda})\big)\bigg\} =0,\nonumber
\end{align}
or 
\begin{align}
2\mathbb{E}(\widehat{ \mathrm{X}}^{\star}_{\lambda}- X)-\lambda\mathbb{E}\bigg\{2\Big(\Vert X-\widehat{ \mathrm{X}}^{\star}_{\lambda}\Vert^{2}_{2}-\mathbb{E}\Vert X-\widehat{ \mathrm{X}}^{\star}_{\lambda}\Vert^{2}_{2} \Big)\big( 2 X-2\mathbb{E} X\big)\bigg\} =0,\nonumber
\end{align}
or 
\begin{align}
\mathbb{E}(\widehat{ \mathrm{X}}^{\star}_{\lambda}- X)-2\lambda\mathbb{E}\bigg\{\Big(\Vert X-\widehat{ \mathrm{X}}^{\star}_{\lambda}\Vert^{2}_{2}-\mathbb{E}\Vert X-\widehat{ \mathrm{X}}^{\star}_{\lambda}\Vert^{2}_{2} \Big)  X\bigg\} =0,\nonumber
\end{align}
or
\begin{align}
\widehat{ \mathrm{X}}^{\star}_{\lambda}=\mathbb{E}\bigg\{ \bigg[1+2\lambda\Big(\Vert X-\widehat{ \mathrm{X}}^{\star}_{\lambda}\Vert^{2}_{2}-\mathbb{E}\Vert X-\widehat{ \mathrm{X}}^{\star}_{\lambda}\Vert^{2}_{2} \Big)\bigg]  X\bigg\}, \label{rdnk}
\end{align}
The Radon-Nikodym derivative in (\ref{rdnk}) takes positive as well as negative values. However, $
\frac{\mathrm{d}{\xi}}{\mathrm{d}\mu}\geq 1-2\lambda \mathbb{E}|| X-\widehat{ \mathrm{X}}^{\star}_{\lambda}||^{2}_{2},\nonumber$
and therefore, 
$
\lambda\leq{1}/{2\mathbb{E}|| X-\widehat{ \mathrm{X}}^{\star}_{\lambda}||^{2}_{2}},~\lambda\in[0,+\infty)$
renders the factor in (\ref{rdnk}) a positive density. This is guaranteed when 
\begin{align}
 \lambda&{\leq}\inf_{\lambda\in[0,+\infty)}\frac{1}{2\mathbb{E}\Vert X-\widehat{ \mathrm{X}}^{\star}_{\lambda}\Vert^{2}_{2}}
 {=}\frac{1}{2\mathbb{E}\Vert X-\widehat{ \mathrm{X}}^{\star}_{\infty}\Vert^{2}_{2}}{=}\frac{1}{2\big(\mathbb{E}\Vert\widehat{\Delta X}^{\star}_{0\infty} \Vert^{2}_{2}+\mathrm{mmse}(\mu)\big)}\equiv{1}/{\gamma^{\star}},\nonumber
\end{align}
where the first equality follows because the objective increases  with $\lambda\geq{0}$.
\end{proof}
\section{Appendix B (Proofs of Sections \ref{drosec}, \ref{generalizationsec})}\label{appB}

\begin{proof}[Proof of Lemma \ref{relationofmultipliers}]
Convexity of (\ref{mv}), as well as of (\ref{md}) and therefore Slater's condition imply that the dual optimal is attained.  The KKT conditions for (\ref{mv}), and (\ref{md}) are 
$\mathrm{d}_{\widehat{ \mathrm{X}}}\mathbb{E}Z_{ \widehat{\mathrm{ X}}^{\star}}=-\lambda^{\star}_{mv}(\varepsilon)\nabla\mathrm{var}Z_{ \widehat{\mathrm{ X}}^{\star}},$ 
and
$
\mathrm{d}_{\widehat{ \mathrm{X}}}\mathbb{E}Z_{ \widehat{\mathrm{ X}}^{\star}}=-\lambda^{\star}_{md}(\varepsilon)\nabla\mathrm{var}Z_{ \widehat{\mathrm{ X}}^{\star}}({2\sqrt{\rule{0pt}{1.2ex}\mathrm{var}Z_{ \widehat{\mathrm{ X}}^{\star}}}})^{-1},$
respectively. Complementary slackness condition for (\ref{md}) yields $
{\lambda^{\star}_{md}(\varepsilon)}=2\lambda^{\star}_{mv}(\varepsilon){\sqrt{\varepsilon}}.
$
\end{proof}

\begin{proof}[Proof of Theorem \ref{Th}]
To keep the notation simple, we use just $\lambda$ for the optimal Lagrange multiplier $\lambda^{\star}_{md}(\varepsilon)$ of the mean-deviation problem, and $\sigma$ for the optimal Lagrange multiplier $\lambda^{\star}_{mv}(\varepsilon)$ of the mean-variance problem. Then, on the basis of Lemma (\ref{relationofmultipliers}) $\lambda=2\sigma\sqrt{\varepsilon}$ and 
\begin{align}
\mathbb{E}_{\nu^{\star}(\lambda)}\Vert X-\widehat{\mathrm{X}}^{\star}_{\sigma}\Vert^{2}_{2}&=\min_{\widehat{\mathrm{X}}}\sup_{\nu \in \mathrm{U}(\lambda)}\mathbb{E}_{\nu}\Vert X-\widehat{\mathrm{X}}\Vert^{2}_{2} \label{ena}\\[5pt]
&=\sup_{\nu \in \mathrm{U}(\lambda)}\mathbb{E}_{\nu}\Vert X-\widehat{\mathrm{X}}^{\star}_{\sigma}\Vert^{2}_{2}\label{dyo}\\[5pt]
&\geq \mathbb{E}_{\nu}\Vert X-\widehat{\mathrm{X}}^{\star}_{\sigma}\Vert^{2}_{2}, ~\nu\in U(\lambda). \label{tria}
\end{align}
Equation (\ref{dyo}) follows from (\ref{ena}) because $\widehat{\mathrm{X}}^{\star}_{\sigma}$ is the optimal solution 
of (\ref{50}). On top of that, 
\begin{align}
\min_{\widehat{\mathrm{X}}}\sup_{\nu \in \mathrm{U}(\lambda)}\mathbb{E}\Vert X-\widehat{\mathrm{X}}\Vert^{2}_{2}&=\mathbb{E}_{\nu^{\star}(\lambda)}\Vert X-\widehat{\mathrm{X}}^{\star}_{\sigma}\Vert^{2}_{2}\label{tesera}\\[5pt]
&=\mathbb{E}_{\nu^{\star}(\lambda)}\Vert X-\mathbb{E}_{\nu^{\star}(\lambda)}X\Vert^{2}_{2},\label{pente}
\end{align}
where (\ref{pente}) follows from (\ref{tesera}) because of Lemma \ref{weightedlemma}. With $\lambda\leq 2\sqrt{\varepsilon}/\gamma^{\star}$, again from Lemma \ref{weightedlemma}, the measure is positive and thus, 
\begin{align}
  \mathbb{E}_{\nu^{\star}(\lambda)}  \Vert X-\mathbb{E}_{\nu^{\star}(\lambda)}X\Vert^{2}_{2}&=\min_{\widehat{\mathrm{X}}}\mathbb{E}_{\nu^{\star}(\lambda)}\Vert X-\widehat{\mathrm{X}}\Vert^{2}_{2}\nonumber\\[5pt]
  &\leq \mathbb{E}_{\nu^{\star}(\lambda)}\Vert X-\widehat{\mathrm{X}}\Vert ^{2}_{2}. \label{exi}
\end{align}
Combining (\ref{tria}) and (\ref{exi}) for $\lambda\leq 2\sqrt{\varepsilon}/\gamma^{\star}$ concludes the proof.
\end{proof}

\begin{proof}[Proof of Theorem \ref{theorem}]
\label{generalization}
It is $\bar{ \mathrm{X}}_{n}=n^{-1}\sum_{i} X_{i}$, $\mathrm{T}_{n}=n^{-1}\sum_{i}( X_{i}-\bar{ \mathrm{X}}_{n})^3$ and $\widehat{ \mathrm{X}}^{\star}_{n,\bar{\lambda}}=\bar{\mathrm{X}}_{n}+\bar{\lambda}\mathrm{T}_{n}$. The condition for strict domination reads
\begin{align}
\bar{\lambda}<\frac{2\mathbb{E}\big[( X- \bar{\mathrm{X}}_{n})\mathrm{T}_{n}\big]}{\mathbb{E}\mathrm{T}^{2}_{n}+\epsilon}=\frac{2n\mathbb{E}\big[( X-n^{-1}\sum_{i} X_{i})\sum_{i}( X_{i}-n^{-1}\sum_{i} X_{i})^3\big]}{\mathbb{E}\big[\big(\sum_{i}( X_{i}-n^{-1}\sum X_{i})^3\big)^2\big]+\epsilon}.\label{condition2}
\end{align}
For the numerator of (\ref{condition2}) we have: 
\begin{align}
    \mathbb{E}\big[( X-\bar{\mathrm{X}}_{n})\sum_{i}( X_{i}-\bar{ \mathrm{X}}_{n})^3\big]=\mathbb{E}\big[ X\sum_{i}( X_{i}-\bar{ \mathrm{X}}_{n})^3\big]-\mathbb{E}\big[\bar{ \mathrm{X}}_{n}\sum_{i}( X_{i}-\bar{ \mathrm{X}}_{n})^3\big] \label{terms}
\end{align}
We compute each term of (\ref{terms}). For the first one we have
\begin{align}
&\mathbb{E}\big[  X\sum_{i}( X_{i}-\bar{ \mathrm{X}}_{n})^3\big]=\mathbb{E}\big[  X\sum_{i}( X_{i}^{3}-3 X_{i}^2\bar{ \mathrm{X}}_{n}+3 X_{i}\bar{ X}^{2}_{n}-\bar{\mathrm{X}}^3_{n})\big]\nonumber\\[5pt]
&=\mathbb{E} X\sum^{n}_{i}\big(\mathbb{E} X^{3}_{i}-3\mathbb{E}[ X^{2}_{i}\bar{ \mathrm{X}}_{n}]+3\mathbb{E}[ X_{i}\bar{ X}^{2}_{n}]-\mathbb{E}\bar{ \mathrm{X}}^{3}_{n}\big). \label{firstte}
\end{align}
We compute each term inside the parenthesis in (\ref{firstte}):
For the first we have $\mathbb{E} X_{i}^{3}=\mathbb{E} X^3$, while for the second: 
\begin{align}
 -3\mathbb{E}[ X_{i}^2\bar{ \mathrm{X}}_{n}]&=-3n^{-1}\mathbb{E}[ X^{2}_{i}( X_{i}+\sum^{n}_{j=1,j\neq{i}} X_{j})] \nonumber\\[5pt]
&=-3n^{-1}\mathbb{E} X^{3}-3n^{-1}(n-1)\mathbb{E} X^{2}\mathbb{E} X.\label{deyteroros}
\end{align}
For the third term inside the parenthesis of (\ref{firstte}) we use 
\begin{align}
\mathbb{E}\bigg(\sum^{n}_{j=1,j\neq{i}} X_{j}\bigg)^2=(n-1)\mathbb{E} X^2+(\mathbb{E} X)^2(n-1)(n-2) \label{quadraticterm}
\end{align}
to obtain
\begin{align}
3\mathbb{E}[ X_{i}\bar{\mathrm{ X}}^{2}_{n}]&=3n^{-2}\mathbb{E} X^3+9n^{-2}(n-1)\mathbb{E} X\mathbb{E} X^{2} +3n^{-2}(n-1)(n-2)(\mathbb{E} X)^3.\label{tritoros}
\end{align}
Lastly, the fourth term reads:
\begin{align}
\mathbb{E}\bar{ \mathrm{ X}}^{3}_{n}=n^{-2} \mathbb{E}  X^3+3 n^{-2}(n-1) \mathbb{E}  X^2 \mathbb{E}  X+n^{-2}(\mathbb{E}  X)^3\left(n^2-3 n+2\right). \label{tetartoros}
\end{align}
Thus, by (\ref{deyteroros}), (\ref{tritoros}), (\ref{tetartoros}), the parenthesis in (\ref{firstte}) reads:
\begin{align}   \big(\mathbb{E} X^{3}_{i}-3\mathbb{E}[ X^{2}_{i}\bar{ \mathrm{X}}_{n}]+3\mathbb{E}[ X_{i}\bar{ \mathrm{ X}}^{2}_{n}]-\mathbb{E}\bar{ \mathrm{ X}}^{3}_{n}\big)
&=\mathbb{E}  X^3\left[2 n^{-2}-3 n^{-1}+1\right] \nonumber\\[5pt]
&\hspace{10pt} +\mathbb{E}  X^2 \mathbb{E} X\left[3 n^{-2}(n-1)\right] \nonumber\\[5pt]
&\hspace{10pt} +(\mathbb{E}  X)^3 2 n^{-2}\left(n^2-3 n+2\right).
\end{align}
The summation in (\ref{firstte}) increases order by $n$:
\begin{align}
 \sum^{n}_{i=1} \big(\mathbb{E} X^{3}_{i}-3\mathbb{E}[ X^{2}_{i}\bar{ \mathrm{X}}_{n}]+3\mathbb{E}[ X_{i}\bar{ \mathrm{ X}}^{2}_{n}]-\mathbb{E}\bar{ \mathrm{ X}}^{3}_{n}\big) &=\mathbb{E}  X^3\left[2 n^{-1}+n-3\right] \nonumber\\[5pt]
& +\mathbb{E}  X^2 \mathbb{E}  X\left[3 n^{-1}(n-1)\right] \nonumber\\[5pt]
&+(\mathbb{E}  X)^3 2 n^{-1}\left(n^2-3 n+2\right).
\end{align}
Therefore, for the first term in (\ref{terms}) we have: 
\begin{align}
\mathbb{E}\big[  X\sum_{i}( X_{i}-\bar{ \mathrm{X}}_{n})^3\big]&=\mathbb{E}  X \mathbb{E}  X^3\left(2 n^{-1}+n-3\right)\nonumber\\[5pt]
&\hspace{15pt} +\mathbb{E}  X^2(\mathbb{E}  X)^23 n^{-1}(n-1) +(\mathbb{E}  X)^4 2 n^{-1}\left(n^2-3 n+2\right)  \nonumber\\[5pt]
&=\mathbb{E} X \big( \mathbb{E} X^3-3\mathbb{E} X^2\mathbb{E} X+2(\mathbb{E} X)^3\big)n^{-1}(n-1)(n-2)\nonumber\\[5pt]
&=\mathbb{E} X\mathbb{E}( X-\mathbb{E} X)^3n^{-1}(n-1)(n-2)
\label{protos_oros} 
\end{align}
Moving forward, for the second term of (\ref{terms}) we have
\begin{align}
&\mathbb{E}\big[\bar{ \mathrm{X}}_{n}\sum_{i}( X_{i}-\bar{ \mathrm{X}}_{n})^3\big]= \mathbb{E}\big[\bar{ \mathrm{X}}_{n}\sum_{i}\left(  X^{3}_{i}-3 X^{2}_{i}\bar{ \mathrm{X}}_{n}+3 X_{i}\bar{ \mathrm{X}}^{2}_{n}-\bar{ \mathrm{X}}^{3}_{n}
\right)\big]\nonumber\\[5pt]
&=\sum_{i}\left(\mathbb{E}[ X^{3}_{i}\bar{ \mathrm{X}}_{n}]-3\mathbb{E}[ X^{2}_{i}\bar{ \mathrm{X}}^2_{n}]+3\mathbb{E}[ X_{i}\bar{\mathrm{X}}^3_{n}]-\mathbb{E}\bar{ \mathrm{X}}^{4}_{n}\right) \label{secondte}
\end{align}
We compute the terms inside the parenthesis in (\ref{secondte}). For the first one we have:
\begin{align}
    \mathbb{E}[ X^{3}_{i}\bar{ \mathrm{X}}_{n}]&=n^{-1}\mathbb{E}[ X_i^3\left( X_i+\sum_{j=1, j \neq i}^n  X_j\right)]\nonumber\\[5pt]
    &=n^{-1}\mathbb{E} X^{4}+n^{-1}(n-1)\mathbb{E} X^{3}\mathbb{E} X. \label{firstoros}
\end{align}
For the second:
\begin{align}
&-3\mathbb{E}[ X^{2}_{i}\bar{ \mathrm{X}}^{2}_{n}]=-3n^{-2}\mathbb{E}[ X^{2}_{i}\left( X_{i}+\sum_{j=1,j\neq{i}} X_{j}\right)^2]
\nonumber\\[5pt]
    &\stackrel{(\ref{quadraticterm})}{=}-3n^{-2} \mathbb{E} X^{4}-6n^{-2}(n-1)\mathbb{E} X^{3}\mathbb{E} X
    -3n^{-2}(n-1) (\mathbb{E}  X^2)^2-3n^{-2}(n-1)(n-2)\mathbb{E} X^{2}(\mathbb{E}  X)^2 \label{secondoros}
\end{align}
The third term in (\ref{secondte}) reads:
\begin{align}
& 3\mathbb{E}[ X_{i}\bar{ \mathrm{X}}^3_{n}]=3n^{-3}\mathbb{E}[ X_{i}(\sum_{j} X_{j})^{3}]\nonumber\\[5pt]
&=3n^{-3}\mathbb{E}[ X_{i}( X_{i}+\sum_{j=1,j\neq{i}} X_{j})^{3}]\nonumber\\[5pt]
&=3n^{-3}\mathbb{E} X^{4}+12n^{-3}(n-1)\mathbb{E} X^{3}\mathbb{E} X+9n^{-3}(n-1)(\mathbb{E} X^{2})^2\nonumber\\[5pt]
&\hspace{15pt}+18n^{-3}(n-1)(n-2)\mathbb{E} X^{2}(\mathbb{E} X)^2+3n^{-3}(\mathbb{E} X)^4(n-1)(n-2)(n-3)
,\label{gms}
\end{align}
where we expanded the third power and subsequently used that
\begin{align}
\mathbb{E}\left(\sum_{j=1,j\neq{i}}^{n} X_{j}\right)^3=(n-1)\mathbb{E} X^{3}+3(n-1)(n-2)\mathbb{E} X^{2}\mathbb{E} X+(\mathbb{E} X)^3(n-1)(n-2)(n-3).\label{gmtspts}
\end{align}
The last term in the parenthesis of (\ref{secondte}) is  
\begin{align}
-\mathbb{E}\bar{ \mathrm{X}}^{4}_{n}&=-n^{-3}\mathbb{E} X^{4}-4n^{-3}(n-1)\mathbb{E} X^{3}\mathbb{E} X-3n^{-3}(n-1)(\mathbb{E} X^{2})^2\nonumber\\[5pt]
&-6n^{-3}(n-1)(n-2)\mathbb{E} X^{2}(\mathbb{E} X)^{2}{-}n^{-3}(n-1)(n-2)(n-3)(\mathbb{E} X)^4,
\end{align}
where we used 
\begin{align}
&\mathbb{E}\bar{ \mathrm{X}}_{n}^4=n^{-4}\sum_{k_{1},\dots,k_{n}\geq{0}}\frac{4!}{k_{1}!\dots k_{n}!}\mathbb{E}[ X^{k_{1}}\cdot\dots\cdot X^{k_{n}}], \label{sum}
\end{align}
where $k_{1},\dots,k_{n}\geq{0}$ are all combinations that sum up to $n$.
By gathering all the terms we have that the parenthesis in (\ref{secondte}) reads:
\begin{align}
\left(\mathbb{E}[ X^{3}_{i}\bar{ \mathrm{X}}_{n}]-3\mathbb{E}[ X^{2}_{i}\bar{ \mathrm{X}}^2_{n}]+3\mathbb{E}[ X_{i}\bar{ \mathrm{X}}^3_{n}]-\mathbb{E}\bar{ \mathrm{X}}^{4}_{n}\right)
&=n^{-3}\mathbb{E} X^{4}(n-1)(n-2)\nonumber\\[5pt]
&+n^{-3}\mathbb{E} X^{3}\mathbb{E} X(n-1)(n-2)(n-4)\nonumber\\[5pt]
&-3(\mathbb{E} X^{2})^2n^{-3}(n-1)(n-2)\nonumber\\[5pt]
&-3\mathbb{E} X^{2}(\mathbb{E} X)^2n^{-3}(n-1)(n-2)(n-4)\nonumber\\[5pt]
&+2(\mathbb{E} X)^4 n^{-3}(n-1)(n-2)(n-3)
\end{align}
Lastly, the outer summation in (\ref{secondte}) increases the order by $n$.
\begin{align}
\mathbb{E}\left[\bar{ \mathrm{X}}_n \sum_i\left( X_i-\bar{ \mathrm{X}}_n\right)^3\right]
&= n^{-2}\mathbb{E} X^{4}(n-1)(n-2)\nonumber\\[5pt]
&+n^{-2}\mathbb{E} X^{3}\mathbb{E} X(n-1)(n-2)(n-4)\nonumber\\[5pt]
&-3(\mathbb{E} X^{2})^2n^{-2}(n-1)(n-2)\nonumber\\[5pt]
&-3\mathbb{E} X^{2}(\mathbb{E} X)^2n^{-2}(n-1)(n-2)(n-4)\nonumber\\[5pt]
&+2(\mathbb{E} X)^4 n^{-2}(n-1)(n-2)(n-3) \nonumber\\[5pt]
&=\mathbb{E} X\mathbb{E}( X-\mathbb{E} X)^3n^{-1}(n-1)(n-2)\nonumber\\[5pt]
&+n^{-2}(n-1)(n-2)\big(\mathbb{E}( X-\mathbb{E} X)^4-3[\mathbb{E}( X-\mathbb{E} X)^2]^2\big)
\label{deyteros_oros}
\end{align}
By plugging (\ref{protos_oros}), and (\ref{deyteros_oros}) into (\ref{terms}) we obtain
\begin{align}
    \mathbb{E}\big[( X-\bar{ \mathrm{X}}_{n})\sum_{i}( X_{i}-\bar{ \mathrm{X}}_{n})^3\big]&=
    \mathbb{E} X\mathbb{E}( X-\mathbb{E} X)^3n^{-1}(n-1)(n-2)\nonumber\\[5pt]&-
    \mathbb{E} X\mathbb{E}( X-\mathbb{E} X)^3n^{-1}(n-1)(n-2)\nonumber\\[5pt]
&-n^{-2}(n-1)(n-2)\big(\mathbb{E}( X-\mathbb{E} X)^4-3[\mathbb{E}( X-\mathbb{E} X)^2]^2\big)
    \label{terms_fin}
\end{align}
or 
\begin{align}
\mathbb{E}\big[( X-\bar{ \mathrm{X}}_{n})\sum_{i}( X_{i}-\bar{ \mathrm{X}}_{n})^3\big]=  n^{-2}(n-1)(n-2)\big(3m^{2}_{2}-m_{4}\big),  
\end{align}
where $m_{2}$, and $m_{4}$ denote the second and fourth central moments of the probability measure $\mu$, respectively. Strictly positive values for $\bar{\lambda}$ are feasible for all those models with $3m^{2}_{2}>m_{4}$. Lastly, the condition (\ref{condition2}) reads:
\begin{align}
\bar{\lambda}&\leq\frac{2\mathbb{E}\big[( X- \bar{\mathrm{X}}_{n})\mathrm{T}_{n}\big]}{\mathbb{E}\mathrm{T}^{2}_{n}+\epsilon}=\frac{2(3m^{2}_{2}-m_{4})}{\mathbb{E}\mathrm{T}^{2}_{n}+\epsilon}\frac{n^{2}-3n+2}{n^3}\nonumber\\[5pt]
&\leq \frac{1}{n}\left[\frac{8(3m^{2}_{2}-m_{4})}{\mathbb{E}\mathrm{T}^{2}_{n}+\epsilon}\right].
\end{align}

For the case where the random variable is supported over a compact subset of $\mathbb{R}^{d}$, we can bound the polynomial terms of $\gamma^{\star}$ in (\ref{weightedlemma}), and also the variance in the denominator of the first factor of (\ref{fisherform}). This assumption is not absolutely necessary and can be avoided by showing that with high probability (\ref{simpempiricalbetter}) is a DRO estimator. Since Theorem \ref{theorem} allows $n\geq 3$, this can be done by controlling the empirical variance in the factor $\lambda/(1+2\lambda\sigma)$ of (\ref{fisherform}) as well as, the polynomial terms in the expectation (w.r.t.\ $\mu_{n}$) related to the threshold $\gamma^{\star}$. Then, for a fixed model, as $\lambda \rightarrow 0$ one has to trade performance for a distributional robustness, the latter being favored by the logarithmic dependency in the number of samples.
\end{proof}

\end{document}